\documentclass[11pt, a4paper]{article}
\usepackage{amsthm}
\usepackage{mathrsfs}
\usepackage[numbers,sort&compress]{natbib}
\usepackage{amsmath,amssymb,latexsym,color}
\usepackage[colorlinks,
linkcolor=blue,
anchorcolor=green,
citecolor=magenta
]{hyperref}
\usepackage{graphicx}
\usepackage{tikz}
\usetikzlibrary{calc}
\usepackage{CJK}

\usepackage{authblk}

\long\def\delete#1{}

\usepackage{color}

\definecolor{Blue}{rgb}{0,0,1}
\definecolor{Red}{rgb}{1,0,0}
\definecolor{DarkGreen}{rgb}{0,0.6,0}
\definecolor{DarkYellow}{rgb}{1,1,0.2}
\definecolor{DarkPurple}{rgb}{.6,0,1}

\usepackage{xcolor}
\usepackage[normalem]{ulem}

\usepackage{enumerate}
\usepackage{cleveref}
\crefformat{section}{\S#2#1#3}
\crefformat{subsection}{\S#2#1#3}
\crefformat{subsubsection}{\S#2#1#3}
\crefrangeformat{section}{\S\S#3#1#4 to~#5#2#6}
\crefmultiformat{section}{\S\S#2#1#3}{ and~#2#1#3}{, #2#1#3}{ and~#2#1#3}
\def\ma{\mathcal{A}}
\def\mb{\mathcal{B}}

\def\mf{\mathcal{F}}
\def\mg{\mathcal{G}}
\def\mh{\mathcal{H}}
\def\mi{\mathcal{I}}
\def\mj{\mathcal{J}}

\def\mp{\mathcal{P}}
\def\mq{\mathcal{Q}}
\def\mr{\mathcal{R}}
\def\ms{\mathcal{S}}
\def\mt{\mathcal{T}}
\def\mv{\mathcal{V}}

\def\bs{\setminus}
\def\mmu{\mathcal{U}}

\def\tf{\tau_t(\mf)}
\def\ttf{\mt_t(\mf)}
\def\tttf{\tau_t(\mt_t(\mf))}

\def\c{\choose}

\def\ge{\geqslant}
\def\le{\leqslant}
\def\b{\brack}

\def\ro{\romannumeral}

\def\a{\alpha}
\def\b{\beta}

\numberwithin{equation}{section}

\allowdisplaybreaks[4]

\newtheorem{thm}{Theorem}[section]
\newtheorem{cl}{Claim}
\newtheorem{lem}[thm]{Lemma}

\newtheorem{pr1}[thm]{Proposition}

\theoremstyle{remark}
\newtheorem{re}{\bf Remark}
\newtheorem{con}{\bf Construction}

\begin{document}
	\setcounter{page}{1}
	\renewcommand{\thefootnote}{}
	\newcommand{\remark}{\vspace{2ex}\noindent{\bf Remark.\quad}}
	\renewcommand{\abovewithdelims}[2]{%
		\genfrac{[}{]}{0pt}{}{#1}{#2}}


	\def\qed{\hfill$\Box$\vspace{11pt}}

\title {\bf Extremal $t$-intersecting families for finite sets with $t$-covering number at least $t+2$}

\author[a]{Tian Yao\thanks{E-mail: \texttt{tyao@hist.edu.cn}}}
\author[b]{Dehai Liu\thanks{Corresponding author. E-mail: \texttt{liudehai@mail.bnu.edu.cn}}}
\author[b]{Kaishun Wang\thanks{E-mail: \texttt{wangks@bnu.edu.cn}}}
\affil[a]{School of Mathematical Sciences, Henan Institute of Science and Technology, Xinxiang 453003, China}
\affil[b]{Laboratory of Mathematics and Complex Systems (Ministry of Education), School of
	Mathematical Sciences, Beijing Normal University, Beijing 100875, China}

\date{}

\openup 0.5\jot
\maketitle

\begin{abstract}
	
	Let $\mf\subseteq{[n]\c k}$ be a $t$-intersecting family. Define the $t$-covering number $\tf$ of $\mf$ as the minimum size of a subset $S$ of $[n]$ with $|S\cap F|\ge t$ for each $F\in\mf$. 
In this paper, we characterize $\mf$ for which $|\mf|$ takes the maximum value under the condition that $\tf\ge t+2$ and $n$ is sufficiently large,  thereby generalizing two results by Frankl.

	\vspace{2mm}
	
	
	\
	

\end{abstract}

\section{Introduction}

For positive integers $n$ and $k$ with $n\ge k$, let $[n]=\{1,2,\dots,n\}$ and ${[n]\c k}$ denote the set of all its $k$-subsets. Suppose that $t$ is a positive integer with $k\ge t$. A family $\mf\subseteq{[n]\c k}$ with $|F\cap F'|\ge t$ for any $F,F'\in\mf$ is called  \emph{$t$-intersecting}, where $t$ is omitted when $t=1$. Moreover, if each member of $\mf$ contains a fixed $t$-subset of $[n]$, then we say $\mf$ is \emph{trivial}, and \emph{non-trivial} otherwise. 
A \emph{$t$-cover} $S$ of $\mf$ is a subset of $[n]$ with $|S\cap F|\ge t$ for each $F\in\mf$. 
The \emph{$t$-covering number} $\tf$ of $\mf$ is defined as the minimum size of a $t$-cover, which is an important parameter to measure the nontriviality  of a $t$-intersecting family. Observe that $t\le\tf\le k$, and $\mf$ is trivial if and only if $\tf=t$.

For $n\ge 2k-t+1$ and $k\ge s\ge t$, define 
$$f(n,k,t,s)=\max\left\{|\mf|:\mf\subseteq{[n]\c k}\ \mbox{is}\ t\mbox{-intersecting with}\ \tf\ge s\right\}.$$
The famous Erd\H{o}s-Ko-Rado Theorem \cite{EKR,Fn,Wn} states that if $n>(t+1)(k-t+1)$, then $f(n,k,t,t)={n-t\c k-t}$ and all  extremal families are trivial. The value of $f(n,k,t,t+1)$  and the structure of maximum-sized non-trivial $t$-intersecting families are determined by the Hilton-Milner-Frankl Theorem \cite{SHMT,SHM} for sufficiently large $n$. 
 In \cite{AK,AK2}, Ahlswede and Khachatrian completely determined $f(n,k,t,s)$ with $s\in\{t,t+1\}$, and characterized extremal families.

There are also extensive results on $f(n,k,1,3)$.
A result of Erd\H{o}s and Lov\'{a}sz \cite{EL1975} yields $f(n,3,1,3)\ge10$. In fact, the equality holds  and all extremal families have been described \cite{T1,Moura1999,Polcyn2017}.  Frankl \cite{Frankl1980} determined $f(n,k,1,3)$ and characterized the unique extremal  structure for $k\ge4$ and sufficiently large $n$. 
Recently, Frankl and Wang \cite{Frankl2025} determined  $f(n,k,1,3)$ for $n\ge2k\ge14$, and
Kupavskii \cite{Kupaviskii2026} showed that Frankl's result in \cite{Frankl1980} holds for $n>2k\ge200$. 
We refer readers to \cite{EL1975,Frankl2023,Frankl1995,Frankl1996,Frudi1988,L1975} for more results on $f(n,k,1,s)$.

For general $t$, Moura \cite{Moura1999} presented recursive constructions of all maximal $t$-intersecting subfamilies of ${[n]\c t+2}$. Very recently,  Frankl \cite{Frankl2026} determined $f(n,t+2,t,t+2)$ and described extremal families. 
In this paper, we  determine $f(n,k,t,t+2)$ and corresponding  extremal families  for $k\ge t+3$ and large $n$. 
To prepare for our main  result, we first  define ${a\c b}=0$ for integers $a$ and $b$ with $a>0>b$ or $b>a>0$, and 
introduce the following three examples of $t$-intersecting families with  $t$-covering number $t+2$.

\begin{con}\label{con-1}
	Let $n>2k$, $k\ge t+3$, $T,A\in{[n]\c t}$ and $B,C\in{[n]\c k-t}$. Suppose $|T\cap A|=t-1$, $T\cap(B\cup C)=\emptyset$ and $A,B,C$ are pairwise disjoint. Pick $u\in C$. Set 
	$$G_1=A\cup B,\quad G_2=A\cup C,\quad G_3=(T\cap A)\cup B\cup\{u\}$$
	and
	\begin{equation*}
	\begin{aligned}
		\mf=&\left\{F\in{[n]\c k}: F\cap(A\cup T)=T,\ F\cap B\neq\emptyset,\ F\cap C\neq\emptyset\right\}\\
		&\cup\left\{F\in{[n]\c k}: A\cup T\subseteq F,\ F\cap(B\cup\{u\})\neq\emptyset\right\}\cup\{G_1,G_2,G_3\}.
			\end{aligned}
	\end{equation*}
	Define
	$$f_1(n,k,t)={n-t\c k-t}-3{n-k-1\c k-t}+{n-k-2\c k-t}+{n-2k+t-1\c k-t}+3.$$
\end{con}

\begin{con}\label{con-2}
	Let $n>2k$, $k\ge t+3$,  $M\in \binom{[n]}{k+2}$ and $W\in \binom{M}{t+2}$. Set
	\begin{equation*}
		\begin{aligned}
			\mf=&\left\{F\in \binom{[n]}{k}: W\subseteq F\right\}\cup\left\{ F\in \binom{[n]}{k}: \left|F\cap W\right|=t+1,\ F\cap \left(M\bs W\right)\neq \emptyset\right\}\\
			&\cup \left\{F\in{M\c k}: |F\cap W|=t\right\}. 
		\end{aligned}
	\end{equation*}
	Define
	$$f_2(n,k,t)={n-t-2\c k-t-2}+(t+2)\left({n-t-2\c k-t-1}-{n-k-2\c k-t-1}\right)+\binom{t+2}{2}.$$
\end{con}

\begin{con}\label{con-3}
	Let $n>2k$, $k\ge t+3$ and  $Z\in{[n]\c t+4}$. Set
	$$\mf=\left\{F\in{[n]\c k}: |F\cap Z|\ge t+2\right\}.$$
	Define
	$$f_3(n,k,t)={t+4\c2}{n-t-4\c k-t-2}+(t+4){n-t-4\c k-t-3}+{n-t-4\c k-t-4}.$$
\end{con}

As stated in Lemmas \ref{con1-maximal}, \ref{con2-maximal} and \ref{con3-maximal}, the families described in Constructions \ref{con-1}, \ref{con-2} and \ref{con-3} have sizes $f_1(n,k,t)$, $f_2(n,k,t)$ and $f_3(n,k,t)$, respectively. 
Our main result is presented as follows, which generalizes two results by Frankl \cite{Frankl1980,Frankl2026}.

\begin{thm}\label{main}
	Let $\mf$ be a $t$-intersecting subfamily of ${[n]\c k}$ with $\tf\ge t+2$. If $k\ge t+3$ and $n\ge{t+3\c2}(k-t+1)^4$,
	then
	$$|\mf|\le\max\{f_1(n,k,t),f_2(n,k,t),f_3(n,k,t)\}.$$
	Moreover, if equality holds, then  $\mf$ is a family described in one of Constructions \ref{con-1}, \ref{con-2} and \ref{con-3}.
\end{thm}

We remark that none of the candidate extremal structures in Theorem \ref{main} are redundant. For example, when $n$ is sufficiently large, the largest families come from Constructions \ref{con-1}, \ref{con-2} and \ref{con-3} when $(k,t)=(14,6)$, $(k,t)=(12,6)$ and $(k,t)=(10,6)$, respectively.

The rest of this paper is organized as follows. In Section \ref{S3}, we mainly investigate minimum $t$-covers of maximal $t$-intersecting families with $t$-covering number $t+2$.  Subsequently, Theorem \ref{main} is proved in Section \ref{S4}.  The proofs of  Lemmas \ref{con1-maximal}, \ref{con2-maximal} and \ref{con3-maximal}, which establish some properties of the families described in Constructions \ref{con-1}--\ref{con-3},  are collected in Section \ref{verify}.

\section{$t$-intersecting families with $t$-covering number $t+2$}\label{S3}

For a $t$-intersecting family $\mf\subseteq{[n]\c k}$ and $A\subseteq[n]$, write $\mf_A=\{F\in\mf: A\subseteq F\}$ and
$$\ttf=\left\{S\subseteq[n]: S\ \mbox{is a}\ t\mbox{-cover of}\ \mf\ \mbox{with size}\ \tf\right\}.$$

\begin{lem}{\rm(\cite[Lemma 2.1]{SHMC})}\label{CST}
Suppose $\mf\subseteq{[n]\c k}$ is a maximal $t$-intersecting family. If $n\ge2k$,  then $\ttf$ is also $t$-intersecting.
\end{lem}

Let $\mf\subseteq{[n]\c k}$ be a maximal $t$-intersecting family with $\tf=t+2$.  
By Lemma \ref{CST}, we have $\tttf\in\{t,t+1,t+2\}$. We divide our investigation into three cases.

\subsection{The case $\tttf=t$}

The family described in Construction \ref{con-1} is an example of this case. We present some properties of such a family here, and verify them in Section \ref{verify-1}.

\begin{lem}\label{con1-maximal}
	Let $n$, $k$, $t$, $T$, $A$, $B$, $C$, $u$, $G_1$, $G_2$, $G_3$ and $\mf$ be as in Construction \ref{con-1}.
	The following hold.
	\begin{itemize}
		\item[\rm(\ro1)]	$\mf$ is a maximal $t$-intersecting subfamily of ${[n]\c k}$.
		\item[\rm(\ro2)]    $\tf=t+2$, $\tttf=t$ and $|\ttf|=(k-t)(k-t+1)+1$.
		\item[\rm(\ro3)]	$|\mf|=f_1(n,k,t)>((k-t)(k-t+1)+1){n-t-2\c k-t-2}-(k-t)(2(k-t)^2+1){n-t-3\c k-t-3}$.
	\end{itemize}
\end{lem}

The following proposition establishes an upper bound for $|\ttf|$.

\begin{pr1}\label{pr1-1}
	Suppose $n>2k$, $k\ge t+3$ and  $\mf$ is a maximal $t$-intersecting subfamily of ${[n]\c k}$ with $\tf=t+2$. If each member of $\ttf$ contains $T\in{[n]\c t}$, then
	$$|\ttf|\le(k-t)(k-t+1)+1.$$
	Moreover, if equality holds, then $\mf\bs\mf_T=\{G_1,G_2,G_3\}$ with
	$$G_1=A\cup B,\quad G_2=A\cup C,\quad G_3=(T\cap A)\cup B\cup\{u\},$$
	where
	$u\in C$, and $A,B,C$ are pairwise disjoint subsets of $[n]$ with $|A|=t$, $|B|=|C|=k-t$, $|T\cap A|=t-1$ and $T\cap(B\cup C)=\emptyset$.
\end{pr1}
\begin{proof}
	Pick $W\in\ttf$ and $F\in\mf\bs\mf_T$. By $|W\cap F|\ge t$ and $|W|=t+2$, we have $|T\cap F|\ge t-2$. If $|T\cap F|=t-2$, then 
	$$\ttf\subseteq\left\{T\cup A: A\in{F\bs T\c 2}\right\}.$$
	This together with $k\ge t+3$ yields
	$$|\ttf|\le{k-t+2\c2}<(k-t)(k-t+1)+1.$$
	In the following, we may assume that $|T\cap G|=t-1$ for each $G\in\mf\bs\mf_T$.
	
	Since $\tf>t+1$, it is routine to check that $|\mf\bs\mf_T|\ge2$. Suppose  $\mf\bs\mf_T=\{G_1,G_2\}$.  
	Notice that $|T\cap G_1|=|T\cap G_2|=t-1$ and $(G_1\cap G_2)\bs T\neq\emptyset$. 
	For each $u\in (G_1\cap G_2)\bs T$, $T\cup\{u\}$ is a $t$-cover of $\mf$. This contradicts the assumption that $\tf=t+2$. Consequently there exist three distinct  members $G_1$, $G_2$ and $G_3$ in $\mf\bs\mf_T$.

	By $|T\cap G_1|=|T\cap G_2|=t-1$, we have
	$$|T\cap G_1\cap G_2|+2(t-1-|T\cap G_1\cap G_2|)\le t.$$
	Then $t-2\le|T\cap G_1\cap G_2|\le t-1$. 
	We divide our following proof into two cases, and set $\alpha=k-|G_1\cap G_2|$.

	\medskip
	\noindent{\bf Case 1.} $|T\cap G_1\cap G_2|=t-2$.
	\medskip
	
	In this case, we have
	$$|T\cap(G_1\cup G_2)|=|T\cap G_1|+|T\cap G_2|-|T\cap G_1\cap G_2|=t.$$
		Since $|T\cap G_1|=|T\cap G_2|=t-1$, there exist $u_1\in G_1\bs G_2$ and $u_2\in G_2\bs G_1$ such that $(T\cap G_1\cap G_2)\cup\{u_1,u_2\}=T$. Since $\tf>t+1$, for each $v\in(G_1\cap G_2)\bs T$, there exists $\varphi(v)\in\mf\bs\mf_T$ such that $v\not\in\varphi(v)$. 
	For convenience, write $\mg=(\mf\bs\mf_T)\bs\{G_1,G_2\}$, and $\Gamma(G)=(G\bs T)\cap(G_1\cap G_2)$, $\gamma(G)=|\Gamma(G)|$ for $G\in\mg$.


	Suppose $\alpha=1$. We claim that 
	 there exists $G_4\in\mg$ and $i\in\{1,2\}$ such that $|T\cap G_i\cap G_4|=t-2$ and $k-|G_i\cap G_4|\ge2$. 	
	From $\tf>t+1$, we obtain $\bigcap_{G\in\mg}\Gamma(G)=\emptyset$. To prove the claim, pick $G_4\in\mg$ such that $\Gamma(G_4)\neq(G_1\cap G_2)\bs T$. Since $\alpha=1$ and $|G_1\cap G_4|, |G_2\cap G_4|\ge t$, we have 
	$\Gamma(G_4)\neq\emptyset$. 
	Recall that $|T\cap G_4|=t-1$ and $(T\cap G_1\cap G_2)\cup\{u_1,u_2\}=T$. Then  $\{u_1,u_2\}\cap G_4\neq\emptyset$. W.l.o.g., suppose $u_1\in G_4$.  We have $(G_2\bs G_4)\cap T\neq\emptyset$ by $|T\cap G_2\cap G_4|=|(T\bs\{u_1\})\cap G_4|=t-2$, and obtain $(G_2\bs G_4)\cap((G_1\cap G_2)\bs T)\neq\emptyset$ from $\Gamma(G_4)\neq(G_1\cap G_2)\bs T$. Thus $|G_2\bs G_4|\ge2$, and the claim holds.

	By the claim above, to get an upper bound on $|\ttf|$, it suffices to consider the case $\alpha\ge2$. 
	Let
	$$\mp_1=\left\{T\cup\{v_1,v_2\}\in\ttf:v_1\in G_1\bs(G_2\cup\{u_1\}),\ v_2\in G_2\bs(G_1\cup\{u_2\})\right\}.$$
	We have $|\mp_1|\le(\alpha-1)^2$. 
	Recall that $\varphi(v)\in\mf\bs\mf_T$ and  $v\not\in\varphi(v)$	for $v\in(G_1\cap G_2)\bs T$. Set
	$$\mp_2=\left\{T\cup\{w_1,w_2\}\in\ttf:w_1\in(G_1\cap G_2)\bs T,\  w_2\in\varphi(w_1)\bs T\right\}.$$
	Then $|\mp_2|\le(k-\alpha-t+2)(k-t+1)$ and $\ttf\subseteq\mp_1\cup\mp_2$. 
	By $2\le\a\le k-t$, we further conclude that
	$$|\ttf|\le|\mp_1|+|\mp_2|\le(k-t)(k-t+1)+1-(\alpha-2)(k-t+1-\alpha)\le(k-t)(k-t+1)+1.$$

	Suppose $|\ttf|=(k-t)(k-t+1)+1$. Then $\alpha=2$, and $|\mp_2|=(k-t)(k-t+1)$.  We claim that, for each $G\in\mg$, $\Gamma(G)\neq\emptyset$. 
	Since $|T\cap G|=t-1$, we may suppose that $u_1\in G$. 
	If $u_2\not\in G$, then $T\cap G_1\cap G_2\subseteq G$, and by $|G\cap G_2|\ge t$, we have $\Gamma(G)\neq\emptyset$.  
	If $u_2\in G$, then $|T\cap G_1\cap G_2\cap G|=t-3$, and by $|G\cap G_2|\ge t$, we also have $\Gamma(G)\neq\emptyset$.

	Since $\tf>t+1$, we know $\bigcap_{G\in\mg}\Gamma(G)=\emptyset$, and may assume that $\Gamma(G_3)\neq(G_1\cap G_2)\bs T$. Then $1\le\gamma(G_3)\le k-t-1$. Set 
	$$\mq_1=\left\{W\in\mp_2: W\cap\Gamma(G_3)=\emptyset\right\},\quad\mq_2=\left\{W\in\mp_2: W\cap\Gamma(G_3)\neq\emptyset\right\}.$$
	Observe that 
	\begin{equation*}
		\begin{aligned}
			\mq_1&\subseteq\{T\cup\{v_1,v_2\}: v_1\in(G_1\cap G_2)\bs(T\cup\Gamma(G_3)),\ v_2\in G_3\bs(T\cup\Gamma(G_3))\},\\
			\mq_2&\subseteq\{T\cup\{w_1,w_2\}: w_1\in\Gamma(G_3),\ w_2\in\varphi(w_1)\bs T\}.
		\end{aligned}
	\end{equation*}
	 Then 
	$$|\mq_1|\le(k-t-\gamma(G_3))(k-t+1-\gamma(G_3)),\quad|\mq_2|\le(k-t+1)\gamma(G_3).$$
	We further conclude that
	$$(k-t)(k-t+1)=|\mp_2|\le|\mq_1|+|\mq_2|\le(k-t)^2+1,$$
	a contradiction. Consequently, $|\ttf|<(k-t)(k-t+1)+1$.

	\medskip
	\noindent{\bf Case 2.} $|T\cap G\cap G'|=t-1$ for any $G,G'\in\mf\bs\mf_T$.
	\medskip
	
	In this case, for each $G\in\mf\bs\mf_T$, by $|T\cap G|=t-1$, we have 
	$$T\cap G=T\cap G\cap G_1=T\cap G_1=T\cap G_1\cap G_2=:S.$$
	Set 
	\begin{equation*}
		\begin{aligned}
			\mr_1&=\left\{R\in\ttf: R\cap G_1\cap G_2=S\right\},\\
			\mr_2&=\left\{R\in\ttf: |R\cap G_1\cap G_2|\in\{t,t+1\}\right\}.
		\end{aligned}
	\end{equation*}
	Observe that $\ttf=\mr_1\cup\mr_2$ and
	\begin{equation*}
		\begin{aligned}
			\mr_1&\subseteq\left\{T\cup U: U\in{G_1\Delta G_2\c 2},\ U\cap G_3\neq\emptyset,\ U\cap (G_i\bs G_j)\neq\emptyset,\ \{i,j\}=\{1,2\}\right\}.
		\end{aligned}
	\end{equation*}
	Write
	$$\b=|G_1\cap G_2\cap G_3|-t+1,\quad x=|G_3\cap(G_1\bs G_2)|,\quad y=|G_3\cap(G_2\bs G_1)|.$$
	Recall that $\a=k-|G_1\cap G_2|$. We have
	\begin{equation}\label{r1r3}
		\begin{aligned}
			|\mr_1|&\le\alpha^2-(\alpha-x)(\alpha-y)=\alpha(x+y)-xy.
		\end{aligned}
	\end{equation}
	Pick $R\in\mr_2$. 
	If $(R\cap G_1\cap G_2\cap G_3)\bs S=\emptyset$, then the number of such $R$ is at most $(k-\alpha-\beta-t+1)(k-\b-t+1)$. 
	If  $w\in (G_1\cap G_2\cap G_3)\bs S$ for some $w\in R$, then by $\tf>t+1$, there exists $G_5\in\mf\bs\mf_T$ with  $w\not\in G_5$, implying that the number of such $R$ is at most $\beta(k-t+1)$. 
	Therefore
	\begin{equation}\label{r2}
		\begin{aligned}
			|\mr_2|&\le(k-\alpha-\beta-t+1)(k-\b-t+1)+\beta(k-t+1).
		\end{aligned}
	\end{equation}
	By \eqref{r1r3} and \eqref{r2}, we have
		\begin{equation}\label{r3}
			|\ttf|=|\mr_1|+|\mr_2|\le(k-t+1)^2-\beta(k-t+1-\beta)-\alpha(k-t+1-\beta-x-y)-xy.
		\end{equation}

	\medskip
	\noindent{\bf Case 2.1.} $\b\neq0$.
	\medskip
	
	By $x+y+\b\le k-t+1$ and $1\le\b\le k-\a-t+1\le k-t$, we have $k-t+1-\beta-x-y\ge0$ and $\beta(k-t+1-\beta)\ge k-t$.  These together with \eqref{r3} yield
	$|\ttf|\le(k-t)(k-t+1)+1$, and if equality holds, then
	$$xy=0,\quad x+y+\b=k-t+1,\quad\b\in\{1,k-t\},$$
	implying that $G_3\in\{G_1,G_2\}$, a contradiction. Therefore $|\ttf|<(k-t)(k-t+1)+1$.
	
	\medskip
	\noindent{\bf Case 2.2.} $\b=0$.
	\medskip

	By $\b=0$ and $|G_3\cap G_1|,|G_3\cap G_2|\ge t$, we have $x,y\ge1$. Then $xy\ge x+y-1$. Observe that $x+y\le k-t+1$. 
	From \eqref{r3}, we further obtain
	\begin{equation*}
		\begin{aligned}
			|\ttf|&\le(k-t+1)^2-\alpha(k-t+1-x-y)-(x+y)+1\\
			&=(k-t+1)^2-\alpha(k-t+1)-(1-\alpha)(x+y)+1\\
			&\le(k-t+1)^2-(k-t+1)+1\\
			&=(k-t)(k-t+1)+1.
		\end{aligned}
	\end{equation*}
	Assume $|\ttf|=(k-t)(k-t+1)+1$. We have $xy=x+y-1$, and either $\alpha=1$ or $x+y=k-t+1$.
	
	Suppose $\a=1$. 
	Then $x=y=1$. 
	Set $G_1\bs G_2=\{p\}$ and $G_2\bs G_1=\{q\}$.  We have  $p,q\in G_3$. Moreover, it follows from \eqref{r2}, \eqref{r3} and $|\ttf|=(k-t)(k-t+1)+1$ that
$$\mr_2=\left\{T\cup\{w_1,w_2\}: w_1\in(G_1\cap G_2)\bs S,\ w_2\in\{p,q\}\cup(G_3\bs(G_1\cup G_2))\right\}.$$
Pick $G\in\mf\bs(\mf_T\cup\{G_1,G_2\})$. 
Note that, prior to the argument in Case 2.1, $G_3$ is an arbitrary member of $\mf\bs(\mf_T\cup\{G_1,G_2\})$. 
Therefore, if $|G\cap G_1\cap G_2|-t+1\neq0$, then by an argument similar to Case 2.1, we obtain $|\ttf|<(k-t)(k-t+1)+1$, a contradiction. 
Now $|G\cap G_1\cap G_2|=t-1$ and hence $G\cap(G_1\cap G_2)=S$. We have $p,q\in G$ by $|G\cap G_1|\ge t$ and $|G\cap G_2|\ge t$. 
It follows from the structure of $\mr_2$ that $G_3\bs(G_1\cup G_2)\subseteq G$. We further obtain $G=G_3$ and $\mf\bs\mf_T=\{G_1,G_2,G_3\}$. 
Write $A= G_1\cap G_3$, $B=(G_1\cap G_2)\bs S$ and $C=\{q\}\cup(G_3\bs(G_1\cup G_2))$. Notice that $S=T\cap A$. We have
$$G_1=A\cup B,\quad G_3=A\cup C,\quad G_2=(T\cap A)\cup B\cup\{q\},$$
as desired.

Suppose $\a\neq1$. We obtain
$$xy=x+y-1,\quad x+y=k-t+1,$$
implying that 
\begin{equation*}\label{g3}
	(x,y)\in\{(1,k-t),(k-t,1)\},\quad\a=k-t.
\end{equation*}
W.l.o.g., assume that $(x,y)=(k-t,1)$. We get $|G_1\cap G_2|=t$, and  $G_3=S\cup(G_1\bs G_2)\cup\{u\}$ for some $u\in G_2\bs G_1$.

	It follows from \eqref{r2}, \eqref{r3} and $|\ttf|=(k-t)(k-t+1)+1$ that 
	$$\mr_2=\{T\cup\{w_1,w_2\}: w_1\in(G_1\cap G_2)\bs S,\ w_2\in(G_1\bs G_2)\cup\{u\}\}.$$
	As in the case $\alpha=1$, an argument similar to Case 2.1 yields $S=G\cap(G_1\cap G_2)$ for any $G\in\mf\bs(\mf_T\cup\{G_1,G_2\})$. 
	This together with $T\cap G=S$ and $|R\cap G|\ge t$ for each $R\in\mr_2$  yields $(G_1\bs G_2)\cup\{u\}\subseteq G$. We further obtain $G=G_3$ and $\mf\bs\mf_T=\{G_1,G_2,G_3\}$. The desired result follows by taking $A=G_1\cap G_2$, $B=G_1\bs G_2$ and $C=G_2\bs G_1$.
\end{proof}

\begin{lem}\label{cor-1}
	Suppose $n>2k$, $k\ge t+3$ and  $\mf$ is a maximal $t$-intersecting subfamily of ${[n]\c k}$ with $\tf=t+2$ and $\tttf=t$. If $|\ttf|=(k-t)(k-t+1)+1$, then $\mf$ is a family described in Construction \ref{con-1}.
\end{lem}
\begin{proof}
	Assume that each member of $\ttf$ contains $T\in{[n]\c t}$. By Proposition \ref{pr1-1}, we know
	$\mf\bs\mf_T=\{G_1,G_2,G_3\}$ with
	$$G_1=A\cup B,\quad G_2=A\cup C,\quad G_3=(T\cap A)\cup B\cup\{u\},$$
	where
	$u\in C$, and $A,B,C$ are pairwise disjoint subsets of $[n]$ with $|A|=t$, $|B|=|C|=k-t$, $|T\cap A|=t-1$ and $T\cap(B\cup C)=\emptyset$. Hence $\mf$ is contained in a family described in Construction \ref{con-1}. 
	The desired result follows from the maximality of $\mf$ and Lemma \ref{con1-maximal}.
\end{proof}

\subsection{The case $\tttf=t+1$}

We first state some properties of the family described in Construction \ref{con-2}, which are proved in Section \ref{verify-2}. 

\begin{lem}\label{con2-maximal}
	Let $n$, $k$, $t$, $M$, $W$ and $\mf$ be as in Construction \ref{con-2}. The following hold.
	\begin{itemize}
		\item[\rm(\ro1)]	$\mf$ is a maximal $t$-intersecting subfamily of ${[n]\c k}$.
		\item[\rm(\ro2)]    $\tf=t+2$,  $\ttf=\{T\in{M\c t+2}: |T\cap W|\ge t+1\}$ and $\tttf=t+1$.
		\item[\rm(\ro3)]	$|\mf|=f_2(n,k,t)>((t+2)(k-t)+1){n-t-2\c k-t-2}-(t+2)(k-t)^2{n-t-3\c k-t-3}$.
	\end{itemize}
\end{lem}

Our main goal in this subsection is to prove the following proposition.
\begin{pr1}\label{pr1-2}
	Suppose $n>2k$, $k\ge t+3$  and $\mf$ is a maximal $t$-intersecting subfamily of ${[n]\c k}$ with $\tf=t+2$ and $\tttf=t+1$. Then at least one of the following holds.
	\begin{itemize}
		\item[\rm{(\ro1)}] $|\ttf|<\max\left\{(k-t)(k-t+1)+1,(t+2)(k-t)+1,{t+4\c2}\right\}$.
			\item[\rm{(\ro2)}] There exist $M\in{[n]\c k+2}$ and $W\in{M\c t+2}$ such that
			$$\ttf=\left\{T\in{M\c t+2}: |T\cap W|\ge t+1\right\}.$$
	\end{itemize}
\end{pr1}

For two families $\ma$ and $\mb$ with $|A\cap B|\ge t$ for any $A\in\ma$ and $B\in\mb$, we say they are \emph{cross $t$-intersecting}. When $t=1$, we omit $t$. 
To prove Proposition \ref{pr1-2}, we need the following theorem.

\begin{thm}{\rm(\cite[Theorem 1.5]{2601191})} \label{2601192}
	Suppose								 $n\ge 2k$ and $r\ge 2$. If $\mf_{1}, \mf_{2}, \ldots ,\mf_{r}\subseteq \binom{[n]}{k}$ are non-empty pairwise cross intersecting families, then 
	$$\sum_{i=1}^{r}\left|\mf_{i}\right|\le \max\left\{ \binom{n}{k}- \binom{n-k}{k}+r-1, r\binom{n-1}{k-1}\right\}.$$
	Moreover, if $n> 2k$ and  $\binom{n}{k}- \binom{n-k}{k}+r-1< r\binom{n-1}{k-1}$, then equality holds if and only if $\mf_{1}=\mf_{2}=\cdots=\mf_{r}=\left\{F\in\binom{[n]}{k}: a\in F\right\}$ for some $a\in [n]$. 
\end{thm}

\begin{re}
	Since $\max\left\{ \binom{n}{k}- \binom{n-k}{k}+r-1, r\binom{n-1}{k-1}\right\}$ is  increasing as  $r$   increases,   the upper bound in Theorem \ref{2601192} remains valid if at least two of  $\mf_{1}, \mf_{2}, \ldots ,\mf_{r}$ are non-empty. 
\end{re}

In the following, denote the family of all minimum $t$-covers of $\mt_{t}(\mf)$ by $\mmu_{t}(\mf)$.

\begin{lem}\label{2601161}
	Suppose $n>2k$, $k\ge t+3$, $t\ge2$  and $\mf$ is a maximal $t$-intersecting subfamily of ${[n]\c k}$ with $\tf=t+2$ and $\tttf=t+1$. If $|U\cap F|\le t-2$ for some $U\in \mmu_{t}(\mf)$ and  $F\in \mf$, then $|\ttf|<(t+2)(k-t)+1$.
\end{lem}

\begin{proof}
	Pick $T\in \mt_{t}(\mf)$. By $\left| T\cap U\right|\ge t$, $\left| T\cap F\right|\ge t$ and $(T\cap U)\cup(T\cap F)\subseteq T$, we have  $T\cap U\cap F=U\cap F$, $\left| U\cap F\right|=t-2$, $|T\cap U|=|T\cap F|=t$ and 
	$$T\cap U \in \left\{\left(U\cap F\right) \cup S: S\in \binom{ U\bs F}{2}\right\},\quad T\bs U\in \binom{F\bs U}{2}. $$
	For each $X\in \binom{ U\bs F}{2}$, write 
	$$\mi(X)=\left\{T\in \mt_{t}(\mf): T\cap U=\left(U\cap F\right)\cup X \right\},\  \mj(X)=\left\{T\bs U: T\in \mi(X) \right\}. $$
	Set $\binom{ U\bs F}{2}=\left\{A, B,C\right\}$. 
	Since $\mt_{t}(\mf)$ is $t$-intersecting, 
	we know $\mj(A), \mj(B), \mj(C)\subseteq \binom{F\bs U}{2}$ are pairwise cross intersecting. 
	It follows from $\mt_{t}(\mf)= \mi(A)\cup \mi(B)\cup \mi(C)$ and $\tttf=t+1$ that at least two of $\mj(A), \mj(B)$ and $\mj(C)$ are non-empty. 
	By Theorem \ref{2601192}, we obtain
	$$\left|\mt_{t}(\mf)\right|\le\left|\mj(A)\right|+ \left|\mj(B)\right|+\left|\mj(C)\right|\le 3(k-t+1)<(t+2)(k-t)+1,$$
	as desired.
\end{proof}

The following auxiliary lemma follows immediately from the definition of cross intersecting families.

\begin{lem}\label{2607151}
	Suppose $n\ge4$ and $r\ge 2$. Let  $a,b,c,d$ be distinct elements of $[n]$. Assume that $\ma_{1},\ma_{2}, \ldots, \ma_{r}\subseteq\binom{[n]}{2}$ are pairwise cross intersecting. 
	\begin{itemize}
		\item[\rm(\ro1)] If $\{a,b\}, \{c,d\}\in\ma_{1}$, then $\ma_{i}\subseteq \{\{a,c\},\{a,d\},\{b,c\},\{b,d\}\}$ for any $i\ge 2$. 
		\item[\rm(\ro2)] If $\{a,b\}, \{a,c\}, \{b,c\}\in \ma_{1}$, then $\ma_{i}\subseteq \{\{a,b\}, \{a,c\}, \{b,c\}\}$ for any $i\ge 2$.
		\item[\rm(\ro3)] If $\{a,c\},\{a,d\}\in\ma_{1}$, then $\ma_{i}\subseteq\{A\in\binom{[n]}{2}: a\in A\}\cup\{\{c,d\}\}$ for any $i\ge 2$.
	\end{itemize}
\end{lem}

\begin{lem}\label{2601222}
	Suppose $n>2k$, $k\ge t+3$  and $\mf$ is a maximal $t$-intersecting subfamily of ${[n]\c k}$ with $\tf=t+2$ and $\tttf=t+1$. If $\mt_{t}(\mf)\nsubseteq \binom{U_{0}\cup F_{0}}{t+2}$ for some $U_{0}\in \mmu_{t}(\mf)$ and $F_{0}\in \mf$ with $\left|U_{0}\cap F_{0}\right|=t-1$, then 
	$$ \left|\mt_{t}(\mf) \right|<\max \left\{(k-t)(k-t+1)+1,(t+2)(k-t)+1,\binom{t+4}{2} \right\} .$$
\end{lem}
\begin{proof}
	By Lemma \ref{2601161}, it is sufficient to consider the case that $|U\cap F|\ge t-1$ for any $U\in\mmu_t(\mf)$ and $F\in\mf$. 
	For $X\subseteq[n]$ with $X\cap U_{0}=\emptyset$, set
	$$\mv_{0}(X)=\left\{ U_{0}\cup \left\{x\right\}: x\in X\right\}\cap \mt_{t}(\mf).$$	
	Write $\binom{U_{0}}{t}=\left\{V_{1},V_{2},\ldots, V_{t+1}\right\}$ and
	$$\mv_{i}=\left\{T\in \mt_{t}(\mf): T\cap U_{0}=V_{i} \right\}, \ \mv_{i}^{\prime}=\left\{T\bs U_{0}: T\in \mv_{i} \right\}, \ i\in [t+1].$$ 
	Observe that $|V_i\cap V_j|=t-1$ for distinct $i,j\in[t+1]$. 
	It follows from Lemma \ref{CST} that $\mv_{1},\mv_{2}, \ldots, \mv_{t+1}$ are pairwise cross $t$-intersecting,  and $\mv_{1}^{\prime}, \mv_{2}^{\prime}, \ldots , \mv_{t+1}^{\prime}$ are pairwise cross intersecting. 
	W.l.o.g., assume that $V_1\cap V_2=U_0\cap F_0$.

	For each $T\in \mt_{t}(\mf)$ with $U_{0}\subseteq T$, by $\left|T\cap F_{0}\right|\ge t$, we derive $T\subseteq U_{0}\cup F_{0}$ and $$\left\{T\in \mt_{t}(\mf):U_{0}\subseteq T\right\}=\mv_{0}(F_{0}\bs U_{0}).$$ 
	Since $U_{0}$ is a $t$-cover of $\mt_{t}(\mf)$, we have 
	\begin{equation}\label{2601223}
		\mt_{t}(\mf)= \mv_{0}(F_{0}\bs U_{0})\cup \mv_{1} \cup \mv_{2}\cup \cdots \cup \mv_{t+1}. 
	\end{equation}
	Observe that each member of $\ttf\bs(\mv_1\cup\mv_2)$ is a subset of $U_0\cup F_0$. 
	By $\ttf\not\subseteq{U_0\cup F_0\c t+2}$, there exists $T_0\in\mv_1\cup\mv_2$ such that
	$T_0\not\subseteq U_0\cup F_0$.
	W.l.o.g., assume $T_{0}\in \mv_{1}$, and $T_{0}\bs U_{0}=\left\{a, b\right\}$ where $a\in F_{0}\bs U_{0}$ and $b\notin U_{0}\cup F_{0}$. For each $i\not\in\{1,2\}$ with $\mv_i\neq\emptyset$, since $\mv'_1$ and $\mv'_i$  are cross intersecting, we obtain
	\begin{equation}\label{2601224}
		\mv_{i}^{\prime}\subseteq \left\{V^{\prime}\in \binom{F_{0}\bs U_{0}}{2}:a\in V^{\prime}\right\}.
	\end{equation}

	\begin{cl}\label{2601225}
		For $E\in\binom{[n]}{t+1}$, we have $\left| \left\{ T\in \mt_{t}(\mf): E\subseteq T\right\}\right|\le k-t+1$.
	\end{cl}
	\begin{proof}
		Since $\tau_{t}(\mf)=t+2$, there exists $F\in\mf$ such that $\left|E\cap F\right|<t$. If $\left|E\cap F\right|\le t-2$, then $\left\{ T\in \mt_{t}(\mf): E\subseteq T\right\}=\emptyset$ and the desired result follows. If $\left|E\cap F\right|= t-1$, then 
		$$|\left\{ T\in \mt_{t}(\mf): E\subseteq T\right\}|\le|\left\{ E\cup \left\{f\right\}: f\in F\bs E\right\}|=k-t+1,$$
		as desired.
	\end{proof}

	It is routine to check that
	 \begin{equation*}\label{2601226}
		(t+1)(k-t)+4<\left\{
		\begin{aligned}
			&(k-t)(k-t+1)+1, &&\mbox{if}\ t=1\ \mbox{and}\ k=t+3,\\
			&\binom{t+4}{2}, &&\mbox{if}\ t\ge 2\ \mbox{and}\ k= t+3,\\
			&(t+2)(k-t)+1, &&\mbox{if}\ k\ge  t+4,
		\end{aligned}
		\right.
	\end{equation*}
	and
	\begin{equation*}\label{2601227}
		3(k-t)+t+2<\left\{
		\begin{aligned}
			&(k-t)(k-t+1)+1, &&\mbox{if}\ t=1,\\
			&\binom{t+4}{2}, &&\mbox{if}\ t\ge 2\ \mbox{and}\ k= t+3,\\
			&(t+2)(k-t)+1, &&\mbox{if}\ t\ge 2\ \mbox{and}\ k\ge  t+4.
		\end{aligned}
		\right.
	\end{equation*}
	To get the desired result, sometimes it is sufficient to show
	$$|\ttf|\le\max\{(t+1)(k-t)+4,3(k-t)+t+2\}.$$
	We divide our following proof into three cases, and remark that, for $i\in[t+1]$ and $x\in F_0\bs U_0$,
	\begin{equation}\label{v0}
		\begin{aligned}
			&\ \mv_0(F_0\bs U_0)\cup\{T\in\ttf: V_i\cup\{x\}\subseteq T\}\\
			=&\ \mv_0(F_0\bs(U_0\cup\{x\}))\cup(\{U_0\cup\{x\}\}\cap\ttf)\cup\{T\in\ttf: V_i\cup\{x\}\subseteq T\}\\
			=&\ \mv_0(F_0\bs(U_0\cup\{x\}))\cup\{T\in\ttf: V_i\cup\{x\}\subseteq T\}.
		\end{aligned}
	\end{equation}

	\medskip
	\noindent
	\textbf{Case 1.} $\mv_{1}^{\prime}\cup \mv_{2}^{\prime}$ is not intersecting.    
	\medskip

	Pick $A$, $B\in \mv_{1}^{\prime}\cup \mv_{2}^{\prime}$ with $A\cap B=\emptyset$. Then $A, B\in \mv_{j_0}^{\prime}$ for some $j_0\in \{1,2\}$. By Lemma \ref{2607151} (\ro1), we get $|\mv_{3-j_0}^{\prime}|\le 4$. For $i\not\in\{1,2\}$, since $a\not\in A$ or $a\not\in B$, $|\mv_i'|\le2$ follows from \eqref{2601224}.

	\medskip
	\noindent
	\textbf{Case 1.1.} $|\mv_{i_{0}}^{\prime}|=2$ for some $i_{0}\notin\{1,2\}$.    
	\medskip
	
	Set $\mv_{i_{0}}^{\prime}=\{\{a,c\}, \{a,d\}\}$, where $c,d\in F_{0}\bs U_{0}$. By Lemma \ref{2607151} (\ro3), we know 
	$$\mv_{j_0}\subseteq \{T\in\mt_{t}(\mf):V_{j_0}\cup \{a\}\subseteq T\}\cup\{ V_{j_0}\cup\{c,d\}\}.$$
	We further conclude $\{A,B\}=\{\{c,d\}, \{a,e\}\}$ for some $e\in [n]$. This together with $\mv_{i_{0}}^{\prime}=\{\{a,c\}, \{a,d\}\}$ and Lemma \ref{2607151} (\ro3)  implies $\mv_{3-j_0}^{\prime}\subseteq \{\{a,c\},\{a,d\}\}$. 
	Consequently, by (\ref{2601223}), (\ref{v0}) and Claim \ref{2601225}, we get
	\begin{equation*}
		|\mt_{t}(\mf)|\le |\mv_{0}(F_{0}\bs U_{0})\cup \mv_{j_0}|+\sum_{i\neq j_0}|\mv_{i}|\le 2(k-t+1)+2t\le (t+1)(k-t)+4,
	\end{equation*}
	as desired.

	\medskip
	\noindent
	\textbf{Case 1.2.} $|\mv_{3-j_0}^{\prime}|\ge 2$ and  $|\mv_{i}^{\prime}|\le 1$ for any $i\notin\{1,2\}$.    
	\medskip
	
	Suppose that $\mv_{3-j_0}^{\prime}$ is not intersecting. By Lemma \ref{2607151} (\ro1), we know $|\mv_{j_0}^{\prime}|\le 4$. From (\ref{2601223}), we obtain 
	$|\mt_{t}(\mf)|\le (k-t+1)+8+t-1=k+8\le 3(k-t)+t+2$.

	Suppose that $\mv_{3-j_0}^{\prime}$ is  intersecting. By Lemma \ref{2607151} (\ro1), we know $\mv_{3-j_0}^{\prime}=\{\{u,c\}:c\in C\}$ for some $u\in A\cup B$ and $C\in\{A, B\}$ with $u\notin C$. It follows from Lemma \ref{2607151} (\ro3) that 
	$$\mv_{j_0}\subseteq \{T\in\mt_{t}(\mf):V_{j_0}\cup \{u\}\subseteq T\}\cup\{ V_{j_0}\cup C\}.$$
	This together with Claim \ref{2601225} yields  $|\mv_{j_0}|\le k-t+2$. From (\ref{2601223}), we get $$|\mt_{t}(\mf)|\le |\mv_{0}(F_{0}\bs U_{0})|+\sum_{i=1}^{t+1}|\mv_{i}|\le 2(k-t)+t+4\le 3(k-t)+t+2,$$
	as required.

	\medskip
	\noindent
	\textbf{Case 1.3.}  $|\mv_{i}^{\prime}|\le 1$ for any $i\neq j_0$.    
	\medskip
	
	There exists $T\in\bigcup_{i\neq j_0}\mv_{i}$ by $\tttf=t+1$. 
	It follows from $|T\cap F_0|\ge t$ that $(T\bs U_0)\cap F_0\neq\emptyset$. 
	Set $T\bs U_{0}=\left\{x,y \right\}$, where $x\in F_0\bs U_0$.  We have
	$$\mv_{j_0}\subseteq \left\{T\in \mt_{t}(\mf): V_{j_0}\cup \left\{x\right\}\subseteq T\right\}\cup  \left\{T\in \mt_{t}(\mf): V_{j_0}\cup \left\{y\right\}\subseteq T\right\}.$$
	This together with \eqref{2601223}, \eqref{v0} and Claim \ref{2601225} yields 
	\begin{equation*}
		\begin{aligned}
			|\ttf|\le|\mv_0(F_0\bs U_0)\cup\mv_{j_0}|+\sum_{i\neq j_0}|\mv_i|\le 3(k-t)+t+2,
		\end{aligned}
	\end{equation*}
	as desired.

	\medskip
	\noindent
	\textbf{Case 2.} $\mv_{1}^{\prime}\cup \mv_{2}^{\prime}$ is non-trivially intersecting.    
	\medskip

	In this case, $\mv_{1}^{\prime}\cup \mv_{2}^{\prime}=\{\{a,b\},\{a,c\},\{b,c\}\}$ for some $c\in[n]$. Recall that $b\not\in F_0$. Either $|(V_1\cup\{b,c\})\cap F_0|$ or $|(V_2\cup\{b,c\})\cap F_0|$ is at least $t$, implying that $c\in F_0\bs U_0$.
	By Lemma \ref{2607151} (\ro2) and (\ref{2601224}), we know $|\mv_{i}|\le1$ for any $i\notin\{1,2\}$. 
	Then $|\mt_{t}(\mf)|\le (k-t+1)+6+t-1\le 3(k-t)+t+2$ follows from \eqref{2601223}.

	\medskip
	\noindent
	\textbf{Case 3.} $\mv_{1}^{\prime}\cup \mv_{2}^{\prime}$ is trivially intersecting.    
	\medskip

	\noindent
	\textbf{Case 3.1.} $\bigcap_{V^{\prime}\in \mv_{1}^{\prime}\cup \mv_{2}^{\prime}}V^{\prime}=\{b\}$.    
	\medskip

	In this case, we have  $\mv_{i}\subseteq \{T\in\mt_{t}(\mf):V_{i}\cup \{b\}\subseteq T\}$ for each $i\in\{1,2\}$.
	Furthermore, there exists $\{b,c\}\in(\mv_{1}^{\prime}\cup \mv_{2}^{\prime})\bs\{\{a,b\}\}$ for some $c\in F_0\bs U_0$. 
	From \eqref{2601224}, for each $j\not\in\{1,2\}$, it follows that  $\mv_{j}^{\prime}\subseteq \{\{a,c\}\}$.  
	This together with \eqref{2601223}  and Claim \ref{2601225}  yields $|\mt_{t}(\mf)|\le 3(k-t+1)+t-1=3(k-t)+t+2$.
	
	\medskip
	\noindent
	\textbf{Case 3.2.} $a\in \bigcap_{V^{\prime}\in \mv_{1}^{\prime}\cup \mv_{2}^{\prime}}V^{\prime}$.    
	\medskip

	In this case, we have $\mv_{i}\subseteq \{T\in\mt_{t}(\mf):V_{i}\cup \{a\}\subseteq T\}$ for each $i\in\{1,2\}$.
	Then it follows from \eqref{2601223}--\eqref{v0} and Claim \ref{2601225} that
	$$|\ttf|\le|\mv_0(F_0\bs U_0)\cup\mv_1|+\sum_{i=2}^{t+1}|\mv_i|\le(t+2)(k-t)+2.$$
	If $t\le2$, then $|\ttf|<\max\{(k-t)(k-t+1)+1,{t+4\c2}\}$, as desired. Next assume  $t\ge3$.

	By \eqref{2601224} and the assumption,  each member of $\bigcup_{i=1}^{t+1}\mv_{i}$ contains $a$. Then  $\binom{U_{0}\cup \left\{a\right\}}{t+1}\subseteq \mmu_{t}(\mf)$. 
	We also obtain 
	$$\mv_{i}=\left\{T\in \mt_{t}(\mf): T\cap \left(U_{0}\cup \left\{a\right\}\right)=V_{i}\cup \left\{a\right\}\right\},\ i\in[t+1]. $$
	Since $\tau_{t}(\mf)=t+2$, there exists $F_{1}\in \mf$ such that $\left|F_{1}\cap \left(V_{1}\cup \left\{a\right\}\right)\right|<t$. Recall that  $\left|U\cap F\right|\ge t-1$ for any $U\in \mmu_{t}(\mf)$ and $F\in \mf$. 
	This together with $\binom{U_{0}\cup \left\{a\right\}}{t+1}\subseteq \mmu_{t}(\mf)$ yields
	$\left|F_{1}\cap \left(V_{1}\cup \left\{a\right\}\right)\right|=t-1$ and $\left|F_{1}\cap \left(U_{0}\cup \left\{a\right\}\right) \right|=t$. 
	Hence $$\mv_{1}\subseteq \left\{V_{1}\cup \left\{a, f\right\}: f\in  F_{1}\bs \left( U_{0}\cup \left\{a\right\}\right)\right\},\quad\left|\mv_{1}\right|\le k-t.$$ Similarly, we have $\left|\mv_{2}\right|\le k-t$. We further conclude  from \eqref{2601224} that
	\begin{equation}\label{2601241}
		\left|\mv_{i}\right|\le k-t,\ i\in[t+1].
	\end{equation}

	Since $V_{1}\cup\left\{a,b\right\}\in \mv_{1}$, we have $b\in F_{1}\bs ( U_{0}\cup \left\{a\right\})$, implying that 
	$$ \left|\left( F_{1}\bs \left( U_{0}\cup \left\{a\right\}\right)\right)\cap \left( F_{0}\bs \left( U_{0}\cup \left\{a\right\}\right)\right)\right|< k-t.$$
	Recall that $t\ge 3$ and $\left|F_{1}\cap \left(U_{0}\cup \left\{a\right\}\right) \right|=t$. There exist at least $3$ members of $\{U_{0}, V_{1}\cup\left\{ a\right\}, \ldots , V_{t+1}\cup\left\{ a\right\}\}$ whose intersections with $F_{1}\cap \left(U_{0}\cup \left\{a\right\}\right)$ have sizes  less than $t$. 
	Suppose that  $Y$ is such a set with $Y\not\in\{ V_{1}\cup \left\{a\right\}, V_{2}\cup \left\{a\right\}\}$.  
	If $Y=U_{0}$, then 
	$$\mv_{0}(F_{0}\bs U_{0})\subseteq \left\{ U_{0}\cup \left\{a\right\}\right\}\cup \left\{U_{0}\cup \left\{f\right\}: f\in \left( F_{1}\bs \left( U_{0}\cup \left\{a\right\}\right)\right)\cap \left( F_{0}\bs \left( U_{0}\cup \left\{a\right\}\right)\right)\right\},$$
	which implies $\left|\mv_{0}(F_{0}\bs U_{0})\right|<k-t+1$. 
	If $Y=V_{i_{0}}\cup \left\{a\right\}$ for some $i_0\notin\left\{1,2\right\}$, then by \eqref{2601224}, we have 
	$$\mv_{i_{0}}\subseteq \left\{V_{i_0}\cup \left\{a,f\right\}: f\in \left( F_{1}\bs \left( U_{0}\cup \left\{a\right\}\right)\right)\cap \left( F_{0}\bs \left( U_{0}\cup \left\{a\right\}\right)\right)\right\},$$
	which implies $\left|\mv_{i_{0}}\right|<k-t$. These together with  \eqref{2601223} and \eqref{2601241} produce $$\left|\mt_{t}(\mf)\right|\le|\mv_0(F_0\bs U_0)|+\sum_{i=1}^{t+1}|\mv_i|\le(k-t+1)+(t+1)(k-t)-1=(t+2)(k-t),$$ as desired.
\end{proof}

\begin{lem}\label{2601221}
	Suppose $n>2k$, $k\ge t+3$ and $\mf$ is a maximal $t$-intersecting subfamily of ${[n]\c k}$ with $\tf=t+2$ and $\tttf=t+1$. If  $\mt_{t}(\mf)\subseteq \binom{U_{0}\cup F_{0}}{t+2}$ for some $U_{0}\in \mmu_{t}(\mf)$ and $F_{0}\in \mf$ with $\left|U_{0}\cap F_{0}\right|=t-1$, then  at least one of the following holds.
		\begin{itemize}
		\item[\rm{(\ro1)}] $|\ttf|<\max\left\{(k-t)(k-t+1)+1,(t+2)(k-t)+1,{t+4\c2}\right\}$.
		\item[\rm{(\ro2)}] There exist $M\in{[n]\c k+2}$ and $W\in{M\c t+2}$ such that
		$$\ttf=\left\{T\in{M\c t+2}: |T\cap W|\ge t+1\right\}.$$
	\end{itemize}
\end{lem}

\begin{proof}
	By Lemma \ref{2601161}, it is sufficient to consider the case that $|U\cap F|\ge t-1$ for any $U\in\mmu_t(\mf)$ and $F\in\mf$. 
	Write $\binom{U_{0}}{t}=\left\{V_{1},V_{2},\ldots, V_{t+1}\right\}$ and  
	$$\mv_{i}=\left\{T\in \mt_{t}(\mf): T\cap U_{0}=V_{i} \right\}, \ \mv_{i}^{\prime}=\left\{T\bs U_{0}: T\in \mv_{i} \right\}, \ i\in [t+1].$$
	Since $U_{0}$ is a $t$-cover of $\mt_{t}(\mf)$, we have
	\begin{equation}\label{2601194}
		\mt_{t}(\mf)=\left\{T\in \mt_{t}(\mf): U_{0}\subseteq T\right\}\cup \mv_{1} \cup \mv_{2}\cup \cdots \cup \mv_{t+1}. 
	\end{equation}
	From $T\subseteq U_{0}\cup F_{0}$ for each $T\in \mt_{t}(\mf)$, we obtain
	\begin{equation}\label{2601195}
		\left\{ T\bs U_{0}: T\in \mt_{t}(\mf),\  U_{0}\subseteq T\right\} \subseteq \binom{F_{0}\bs U_{0} }{1},\quad\mv_{1}^{\prime}, \mv_{2}^{\prime}, \ldots, \mv_{t+1}^{\prime}\subseteq \binom{F_{0}\bs U_{0}}{2}.
	\end{equation}
	Then $$\left|\left\{T\in \mt_{t}(\mf): U_{0}\subseteq T\right\}\right|\le k-t+1.$$
	Since $\mt_{t}(\mf)$ is $t$-intersecting and $\tttf=t+1$, we know 	 $\mv_{1}^{\prime}, \mv_{2}^{\prime}, \ldots, \mv_{t+1}^{\prime}$ are pairwise cross intersecting and  at least two of them are non-empty. 
	It follows from (\ref{2601195}) and Theorem \ref{2601192} that
	$$\left|\mv_{1}^{\prime}\right|+\left|\mv_{2}^{\prime}\right|+\cdots +\left|\mv_{t+1}^{\prime}\right|\le \max\{2k-t-1,(t+1)(k-t)\}.$$
	
	If $t=1$, then $2k-t-1=(t+1)(k-t)<(k-t-1)(k-t+1)+1$. If $t\ge 2$ and $k=t+3$, then $2k-t-1<(t+1)(k-t)<{t+4\c2}-(k-t+1)$.
	In a word, if $t=1$, or $t\ge2$ and $k=t+3$, then (\ro1) holds.

	Now suppose  $k\ge t+4\ge 6$. We have
	$$|\ttf|\le|\left\{T\in \mt_{t}(\mf): U_{0}\subseteq T\right\}|+\sum_{i=1}^{t+1}|\mv_i|\le(t+2)(k-t)+1.$$
	Assume that $|\ttf|=(t+2)(k-t)+1$. We have 
	$$\left\{ T\bs U_{0}: T\in \mt_{t}(\mf),\  U_{0}\subseteq T\right\} = \binom{F_{0}\bs U_{0} }{1},$$
	and by Theorem \ref{2601192} and $2k-t-1<(t+1)(k-t)$, 
	$$\mv_1'=\mv_2'=\cdots=\mv'_{t+1}=\left\{V\in\binom{F_{0}\bs U_{0}}{2}: a \in V\right\}$$
	for some $a\in F_{0}\bs U_{0}$.
	These together with  \eqref{2601194} yield
	$$\mt_{t}(\mf)=\left\{T\in \binom{ U_{0}\cup F_{0}}{t+2}: \left| T\cap \left( U_{0}\cup \left\{a\right\}\right)\right|\ge t+1\right\},$$
	as desired.
\end{proof}

\begin{proof}[\bf Proof of Proposition \ref{pr1-2}] 
	Pick $U_0\in\mmu_t(\mf)$. 
Since $\tf=t+2$ and $\tttf=t+1$, there exists $F_0\in\mf$ such that $|U_0\cap F_0|<t$.  
If $|U_0\cap F_0|\le t-2$, then $t\ge2$, and (\ro1) holds by Lemma \ref{2601161}. 
Suppose $|U_0\cap F_0|=t-1$. 
If $\mathcal{T}_t(\mathcal{F}) \nsubseteq
\binom{U_0 \cup F_0}{t+2}$, then (\ro1) follows from Lemma  \ref{2601222}; in the opposite case, Lemma \ref{2601221} implies that at least one of (\ro1) and (\ro2) holds.
\end{proof}

\begin{lem}\label{cor-2}
	Suppose $n>2k$, $k\ge t+3$  and $\mf$ is a maximal $t$-intersecting subfamily of ${[n]\c k}$ with $\tf=t+2$ and $\tttf=t+1$. If $\ttf=\{T\in{M\c t+2}: |T\cap W|\ge t+1\}$ for some $M\in{[n]\c k+2}$ and $W\in{M\c t+2}$, then $\mf$ is a family described in Construction \ref{con-2}.
\end{lem}
\begin{proof}
	For $F\in\mf$, we have $\left|F\cap W\right|\ge t$ by $W\in \mt_{t}(\mf)$. Set 
	$$\ms_i=\{F\in\mf:|F\cap W|=t+i\},\quad i\in\{0,1,2\}.$$
	
	Suppose $F\in\ms_0$.
	Then there exists $U\in\binom{W}{t+1}$ such that $\left|F\cap U\right|=t-1$. By $U\cup \left\{m\right\}\in \mt_{t}(\mf)$  for each $m\in M\bs W$, we know $M\bs W\subseteq F$, which implies 
		\begin{equation}\label{F1}
			\ms_0\subseteq\left\{F\in{[n]\c k}: |F\cap W|=t,\ M\bs W\subseteq F\right\}=\left\{F\in{M\c k}: |F\cap W|=t\right\}.
			\end{equation} 
			Suppose $F\in\ms_1$. 
			By $\tau_{t}(\mf)=t+2$, there exists $G\in\mf$ such that $\left|G\cap F\cap W\right|<t$. This together with $\left|G\cap W\right|\ge t$ yields $\left|G\cap W\right|=t$. Notice that $G=\left(G\cap W\right)\cup \left(M\bs W\right)$. Since $\mf$ is $t$-intersecting, we obtain $F\cap  \left(M\bs W\right)\neq \emptyset$ and
			\begin{equation}\label{F2}
				\ms_1\subseteq\left\{F\in{[n]\c k}: |F\cap W|=t+1,\ F\cap(M\bs W)\neq\emptyset\right\}.
			\end{equation} 
	Note that
	\begin{equation}\label{F3}
		\ms_2\subseteq\left\{F\in{[n]\choose k}: W\subseteq F\right\}.	
	\end{equation} 
	By \eqref{F1}--\eqref{F3}, $\mf$ is contained in a family described in Construction \ref{con-2}. 
	Then the desired result follows from the maximality of $\mf$ and Lemma \ref{con2-maximal}.
\end{proof}

\subsection{The case $\tau_t(\ttf)=t+2$}

We begin with some properties of the family described in Construction \ref{con-3}, which are verified in Section \ref{verify-3}.

\begin{lem}\label{con3-maximal}
	Let $n$, $k$, $t$, $Z$ and $\mf$ be as in Construction \ref{con-3}. The following hold.
	\begin{itemize}
		\item[\rm(\ro1)] $\mf$ is a maximal $t$-intersecting subfamily of ${[n]\c k}$.
		\item[\rm(\ro2)] $\tf=t+2$, $\ttf={Z\c t+2}$ and $\tttf=t+2$.
		\item[\rm(\ro3)] $|\mf|=f_3(n,k,t)>{t+4\c2}({n-t-2\c k-t-2}-2{n-t-3\c k-t-3})$.
	\end{itemize}
\end{lem}

\begin{pr1}\label{pr1-3}
	Suppose $n>2k$, $k\ge t+3$  and $\mf$ is a maximal $t$-intersecting subfamily of ${[n]\c k}$ with $\tf=t+2$ and $\tttf=t+2$. 
	\begin{itemize}
		\item[{\rm(\ro1)}] If $t=1$, then $|\ttf|<(k-t)(k-t+1)+1$.
			\item[{\rm(\ro2)}] If $t\ge2$, then $|\ttf|\le{t+4\c2}$, and if equality holds, then $\ttf={Z\c t+2}$ for some $Z\in{[n]\c t+4}$.
	\end{itemize}
\end{pr1}
\begin{proof}
	By Lemma \ref{CST}, we know $\ttf\subseteq{[n]\c t+2}$ is a $t$-intersecting family with $\tttf=t+2$. Then the desired result follows from \cite[Theorem 1.8]{Frankl2026}, which states that  if $\mh\subseteq{[n]\c t+2}$ is a $t$-intersecting family with $\tau_t(\mh)=t+2$, then $|\mh|\le{t+4\c2}$ and when $t\ge2$, equality holds if and only if $\mh={Y\c t+2}$ for some $Y\in{[n]\c t+4}$.
\end{proof}

\begin{lem}\label{cor-3}
	Suppose $n>2k$, $k\ge t+3$  and $\mf$ is a maximal $t$-intersecting subfamily of ${[n]\c k}$ with $\tf=t+2$ and $\tttf=t+2$. If  $\ttf={Z\c t+2}$ for some $Z\in{[n]\c t+4}$, then  $\mf$ is a family described in Construction \ref{con-3}.
\end{lem}
\begin{proof}
	If $|F\cap Z|\le t+1$ for some $F\in\mf$, then there exists $T\in\ttf$ such that
	$$|T\cap F|=|T\cap F\cap Z|\le t-1.$$
	This contradicts the fact that $T$ is a $t$-cover of $\mf$. Hence $|F\cap Z|\ge t+2$ for each $F\in\mf$. 
	Then the desired result follows from the maximality of $\mf$ and Lemma \ref{con3-maximal}. 
\end{proof}

\section{Proof of Theorem \ref{main}}\label{S4}

Before proving Theorem \ref{main}, we first show some results related to upper bounds on sizes of $t$-intersecting families.

\begin{lem}{\rm(\cite[Lemma 2.7]{SHMC})}\label{FS}
	Suppose $n\ge2k$. Let $\mf\subseteq{[n]\c k}$, $S\in{[n]\choose s}$ and $G\in{[n]\c k}$ with $|G\cap S|=\ell<t$ and  $|G\cap F|\ge t$ for each $F\in\mf$. There exists $R\in{[n]\choose s+t-\ell}$ with $S\subseteq R$ such that
	$$|\mf_S|\le{k-\ell\choose t-\ell}|\mf_R|.$$	
\end{lem}

\begin{lem}\label{bigf}
	Suppose $n\ge\max\{2k,(k-t)(k-t+1)+t\}$, $k\ge t+3$ and $\mf\subseteq{[n]\c k}$ is a $t$-intersecting family. Then 
	\begin{equation}\label{bigF}
		|\mf|\le (k-t+1)^{\tau_t(\mf)-t}{\tau_t(\mf)\choose t}{n-\tau_t(\mf)\choose k-\tf}.
	\end{equation}
	If each member of $\mb\subseteq\mf$ does not contain any member of $\ttf$, then 
	\begin{equation}\label{bigB}
		|\mb|\le(k-t+1)^{\tf-t+1}{\tf\c t}{n-\tf-1\c k-\tf-1}.
			\end{equation}
\end{lem}

\begin{proof}
	If $\tau_t(\mf)=t$, then there is nothing to prove. Hence we may assume that $\tau_t(\mf)\ge t+1$. 
	Let $\mg\subseteq\mf$ with $\mg\neq\emptyset$, and $S\in\ttf$. We have 
	\begin{equation}\label{ga01}
		\mg=\bigcup_{A\in{S\c t}}\mg_A.
	\end{equation}

	Pick $A_0\in{S\c t}$ with $\mg_{A_0}\neq\emptyset$. Since $|A_0|<\tau_t(\mf)$, we have $|A_0\cap F_0|<t$ for some $F_0\in\mf$. The set $F_0$ is a $t$-cover of $\mg$. 
	By Lemma \ref{FS}, there exists a $(|A_0|+t-|A_0\cap F_0|)$-subset $A_1$ of $[n]$ with $A_0\subsetneq A_1$ and
	$$|\mg_{A_0}|\le{k-|A_0\cap F_0|\c  t-|A_0\cap F_0|}|\mg_{A_1}|\le(k-t+1)^{|A_1|-|A_0|}|\mg_{A_1}|.$$
	If $|A_1|<\tau_t(\mf)$, then apply Lemma \ref{FS} on $A_1$. 
	Using Lemma \ref{FS} repeatedly, 
	we finally get a series of subsets $A_0,A_1,\dots,A_\ell$ of $[n]$ with $|A_0|<|A_1|<\dots<|A_{\ell-1}|<\tau_t(\mf)\le|A_\ell|$ and  $|\mg_{A_{i}}|\le(k-t+1)^{|A_{i+1}|-|A_i|}|\mg_{A_{i+1}}|$
	for each $i\in\{0,1,\dots,\ell-1\}$. 
	It follows that
		\begin{equation}\label{g0l}
			|\mg_{A_0}|\le(k-t+1)^{|A_\ell|-t}|\mg_{A_\ell}|\le(k-t+1)^{|A_\ell|-t}{n-|A_\ell|\c k-|A_\ell|}.	
			\end{equation}

		Set $\mg=\mf$. 
	Since $n\ge(k-t)(k-t+1)+t$, for $x\in\{t,t+1,\dots,k-1\}$, we have
	\begin{equation}\label{ga03}
		{n-x\c k-x}=\dfrac{n-x}{k-x}{n-x-1\c k-x-1}\ge\dfrac{n-t}{k-t}{n-x-1\c k-x-1}\ge(k-t+1){n-x-1\c k-x-1}.
	\end{equation}
	Notice that $|A_\ell|\le k$ from \eqref{g0l} and $\mf_{A_0}\neq\emptyset$. By \eqref{ga01}--\eqref{ga03}, we get \eqref{bigF}. 
	
	Suppose $\tf=k$. Then $\mf\subseteq\ttf$ and $\mb=\emptyset$, yielding the desired result. In the following, assume $\tf<k$. If $\mb=\emptyset$, then \eqref{bigB} is immediate. Next we consider the case  $\mb\neq\emptyset$. 
	Set $\mg=\mb$. 
	We claim that
	\begin{equation}\label{addclaim}
		|\mb_{A_0}|\le(k-t+1)^{\tf-t+1}{n-\tf-1\c k-\tf-1}.
	\end{equation}
	Suppose $|A_\ell|\ge\tf+1$. We obtain $|A_\ell|\le k$ from \eqref {g0l} and $\mb_{A_0}\neq\emptyset$. Then \eqref{addclaim} follows from \eqref{g0l} and \eqref{ga03}. Now suppose $|A_\ell|=\tf$. 
	Since $\mb_{A_\ell}\neq\emptyset$,  by the definition of $\mb$, we have $A_\ell\not\in\ttf$. 
	Thus $|A_\ell\cap F_\ell|<t$ for some $F_\ell\in\mf$. 
	By Lemma \ref{FS}, there exists a subset $A_{\ell+1}$ of $[n]$ with $|A_{\ell+1}|=\tf+t-|A_\ell\cap F_\ell|$ and
	\begin{equation}\label{add}
		|\mb_{A_\ell}|\le{k-|A_\ell\cap F_\ell|\c t-|A_\ell\cap F_\ell|}|\mb_{A_{\ell+1}}|\le(k-t+1)^{|A_{\ell+1}|-\tf}|\mb_{A_{\ell+1}}|.
	\end{equation}
	Since $\mb_{A_\ell}\neq\emptyset$, we have $\mb_{A_{\ell+1}}\neq\emptyset$ and $|A_{\ell+1}|\le k$. 
	It follows from \eqref{ga03}, \eqref{add}, $\tf<|A_{\ell+1}|$  and
	$|\mb_{A_{\ell+1}}|\le{n-|A_{\ell+1}|\c k-|A_{\ell+1}|}$ that
	\begin{equation*}
		|\mb_{A_\ell}|\le(k-t+1){n-\tf-1\c k-\tf-1}.
	\end{equation*}
	This together with \eqref{g0l} completes the proof of \eqref{addclaim}. Finally, \eqref{bigB} follows from \eqref{ga01} and \eqref{addclaim}.
\end{proof}

\begin{proof}[\bf Proof of Theorem \ref{main}] 
	Suppose that $\mf$ is a $t$-intersecting subfamily  of ${[n]\c k}$ with maximum size under the condition that $\tf\ge t+2$. 
	Since $\tau_t(\mf)\le\tau_t(\mf\cup\{H\})$ for each $H\in{[n]\c k}\bs\mf$, $\mf$ is a maximal $t$-intersecting family.
	Recall that $f_1(n,k,t)$, $f_2(n,k,t)$ and $f_3(n,k,t)$ are defined in Constructions \ref{con-1}, \ref{con-2} and \ref{con-3}, respectively. 
	From Lemmas \ref{con1-maximal}, \ref{con2-maximal} and \ref{con3-maximal}, we further obtain
	\begin{equation}\label{contra}
		|\mf|\ge\max\{f_1(n,k,t),f_2(n,k,t),f_3(n,k,t)\}.
	\end{equation}
	 It is sufficient to show that $\mf$ is a family described in one of Constructions \ref{con-1}, \ref{con-2} and \ref{con-3}.
	
	By Lemma \ref{bigf}, we have $|\mf|\le g(n,k,t,\tf)$ where
	$$g(n,k,t,x):=(k-t+1)^{x-t}{x\c t}{n-x\c k-x}.$$
	For $x\in\{t+3,\cdots,k-1\}$, by $n\ge{t+3\c2}(k-t+1)^4$, we have
	$$\dfrac{g(n,k,t,x+1)}{g(n,k,t,x)}=(k-t+1)\cdot\dfrac{x+1}{x-t+1}\cdot\dfrac{k-x}{n-x}\le\dfrac{(t+4)(k-t+1)(k-t-3)}{4(n-t-3)}<1.$$
	Suppose $\tf\ge t+3$. Then $|\mf|\le(k-t+1)^3{t+3\c3}{n-t-3\c k-t-3}$. By $n\ge{t+3\c2}(k-t+1)^4$ and Lemma \ref{con3-maximal} (\ro3), we have
	\begin{equation*}
		\begin{aligned}
			\dfrac{f_3(n,k,t)-|\mf|}{{t+4\c2}{n-t-3\c k-t-3}}>&\dfrac{n-t-2}{k-t-2}-2-\dfrac{(t+1)(t+2)(k-t+1)^3}{3(t+4)}\\
			\ge&\ \dfrac{1}{k-t-2}\left(n-t-2-(2+(t+1)(k-t+1)^3)(k-t-2)\right)\\
			\ge&\ \dfrac{1}{k-t-2}(n-(t+2)(k-t+1)^4)>0,
		\end{aligned}
	\end{equation*}
	a contradiction to \eqref{contra}. Therefore $\tf=t+2$. 
	
	Note that $\tttf\in\{t,t+1,t+2\}$  from  Lemma \ref{CST} and the maximality of $\mf$, and
	\begin{equation}\label{plus}
		\mf=\left\{F\in\mf: T\subseteq F\ \mbox{for some}\ T\in\ttf\right\}\cup\left\{F\in\mf: T\not\subseteq F\ \mbox{for each}\ T\in\ttf\right\}.
	\end{equation}
	We claim that 
	\begin{equation}\label{ttf}
		|\ttf|\ge\max\left\{(k-t)(k-t+1)+1,(t+2)(k-t)+1,{t+4\c2}\right\}.
	\end{equation}
	If $|\ttf|<(k-t)(k-t+1)+1$, then by \eqref{plus}, $n\ge{t+3\c2}(k-t+1)^4$, Lemmas \ref{con1-maximal} (\ro3) and \ref{bigf}, we obtain
	\begin{equation*}
		\begin{aligned}
			\dfrac{f_1(n,k,t)-|\mf|}{{n-t-3\c k-t-3}}&>\dfrac{n-t-2}{k-t-2}-(k-t)(2(k-t)^2+1)-(k-t+1)^3{t+2\c2}\\
			&>\dfrac{1}{k-t-2}\left(n-t-2-(k-t-2)(k-t+1)^3{t+3\c2}\right)>0,
		\end{aligned}
	\end{equation*}
	a contradiction to \eqref{contra}. If $|\ttf|<(t+2)(k-t)+1$, then by \eqref{plus}, $n\ge{t+3\c2}(k-t+1)^4$, Lemmas \ref{con2-maximal} (\ro3) and \ref{bigf}, we obtain
	\begin{equation*}
		\begin{aligned}
			\dfrac{f_2(n,k,t)-|\mf|}{{n-t-3\c k-t-3}}&>\dfrac{n-t-2}{k-t-2}-(t+2)(k-t)^2-(k-t+1)^3{t+2\c2}\\
			&>\dfrac{1}{k-t-2}\left(n-t-2-(k-t-2)(k-t+1)^3{t+3\c2}\right)>0.
		\end{aligned}
	\end{equation*}
	This contradicts \eqref{contra}.
	If $|\ttf|<{t+4\c2}$, then by \eqref{plus},  $n\ge{t+3\c2}(k-t+1)^4$, Lemmas \ref{con3-maximal} (\ro3) and \ref{bigf}, we obtain
	\begin{equation*}
		\begin{aligned}
			\dfrac{f_3(n,k,t)-|\mf|}{{n-t-3\c k-t-3}}&>\dfrac{n-t-2}{k-t-2}-2{t+4\c2}-(k-t+1)^3{t+2\c2}\\
			&>\dfrac{n-t-2}{k-t-2}-\dfrac{10}{3}{t+3\c2}-(k-t+1)^3{t+3\c2}\\
			&>\dfrac{1}{k-t-2}\left(n-(k-t-2)((k-t+1)^3+4){t+3\c2}\right)>0.
		\end{aligned}
	\end{equation*}
	This also  contradicts \eqref{contra}. We further conclude that \eqref{ttf} holds.
	
	If $\tttf=t$, then by \eqref{ttf}, Proposition \ref{pr1-1} and Lemma \ref{cor-1}, $\mf$ is a family described in Construction \ref{con-1}.  
	If $\tttf=t+1$, then by \eqref{ttf}, 	Proposition \ref{pr1-2} and Lemma \ref{cor-2}, $\mf$ is a family described in Construction \ref{con-2}. 
	If $\tttf=t+2$, then by \eqref{ttf}, Proposition \ref{pr1-3} and Lemma \ref{cor-3}, $\mf$ is a family described in Construction \ref{con-3}. This finishes our proof.	
\end{proof}

\section{Proofs of Lemmas \ref{con1-maximal}, \ref{con2-maximal} and \ref{con3-maximal}}\label{verify}
\subsection{Proof of Lemma \ref{con1-maximal}}\label{verify-1}
	Write
	\begin{align*}
		\ma&=\left\{F\in{[n]\c k}: A\cup T\subseteq F,\ F\cap(B\cup\{u\})\neq\emptyset\right\},\\
		\mb&=\left\{F\in{[n]\c k}: F\cap(A\cup T)=T,\ F\cap B\neq\emptyset,\ F\cap C\neq\emptyset\right\}.	
	\end{align*}

	(\ro1) It is routine to check that both $\ma\cup\mb$ and $\{G_1,G_2,G_3\}$ are $t$-intersecting. Pick $F_0\in\ma\cup\mb$. If $F_0\in\ma$, then by $A\subseteq F_0$, we have $|F_0\cap G_i|\ge|A|=t$ for $i\in\{1,2\}$, and $|F_0\cap G_3|\ge|T\cap A|+|F_0\cap(B\cup\{u\})|\ge t$. On the other hand, if $F_0\in\mb$, then $|F_0\cap G_i|\ge|T\cap A|+1=t$ for $i\in\{1,2\}$, and $|F_0\cap G_3|\ge|T\cap A|+|F_0\cap B|\ge t$. We further conclude that $\mf$ is $t$-intersecting.

	Pick $H\in{[n]\c k}\bs\mf$. To show $\mf$ is maximal, it is sufficient to prove that $\mf\cup\{H\}$ is not $t$-intersecting.

	Suppose $T\subseteq H$. If $A\cup T\subseteq H$, then by $H\not\in\ma$, we have $H\cap(B\cup\{u\})=\emptyset$ and $|H\cap G_3|<t$. If $(A\cup T)\cap H=T$, then by $H\not\in\mb$, either $H\cap B$ or $H\cap C$ is empty, and $|H\cap G_i|<t$ for some $i\in\{1,2\}$. 
	
	Suppose $T\not\subseteq H$. 
	Assume $H\cap B=\emptyset$. By $H\neq G_2$, we know $A\not\subseteq H$ or $C\not\subseteq H$. If $A\not\subseteq H$, then $|G_1\cap H|<t$. If $C\not\subseteq H$, then $F_1=T\cup(B\bs\{p\})\cup\{q\}\in\mb$, where $p\in B$ and $q\in C\bs H$, satisfies $|H\cap F_1|=|H\cap T|<t$. Similarly, if $H\cap C=\emptyset$, then $\mf\cup\{H\}$ is not $t$-intersecting.

	Now assume that neither $H\cap B$ nor $H\cap C$ is empty. If $B\not\subseteq H$ and $C\not\subseteq H$, then pick $b\in B\bs H$ and $c\in C\bs H$. Since $n>2k$, we have $|[n]\bs(H\cup A\cup T)|\ge (2k+1)-(k+t+1)=k-t$. Let $D\in{[n]\bs(H\cup A\cup T)\c k-t}$ with $b,c\in D$. We obtain $F_2=T\cup D\in\mb$ and $|H\cap F_2|=|H\cap T|<t$.

	Suppose $C\subseteq H$.  By $H\cap B\neq\emptyset$, we know that $F_3=(A\cup T)\cup(B\bs\{y\})\in\ma$, where $y\in H\cap B$, satisfies $|H\cap F_3|\le|H\bs(C\cup\{y\})|=t-1$. 
	
	Suppose $B\subseteq H$. When there exists $z\in(H\bs(B\cup\{u\}))\cap C$, we have $F_4=(T\cup A)\cup(C\bs\{z\})\in\ma$ and $|H\cap F_4|\le|H\bs(B\cup\{z\})|=t-1$. 
	Next assume $(H\bs(B\cup\{u\}))\cap C=\emptyset$. 
	This together with $H\cap C\neq\emptyset$ yields $u\in H$.  We have $A\cap T\not\subseteq H$ by $B\cup\{u\}\subseteq H$ and $H\neq G_3$. Therefore, $A\cap T\neq\emptyset$ and $t\ge2$.  
	If $|H\cap T|\le t-2$, then $F_5=T\cup(C\bs\{u\})\cup\{u'\}\in\mb$, where $u'\in B$, satisfies $|H\cap F_5|\le|H\cap T|+1<t$. 
	Since $|H\cap(A\cup T)|\le t-1$ and $A\cap T\not\subseteq H$,  
	if $|H\cap T|=t-1$, then  $|H\cap A|\le t-2$ and $|H\cap G_2|\le|H\cap A|+1<t$.
	
	In summary, $\mf\cup\{H\}$ is not $t$-intersecting, as desired.

	(\ro2) The set $T\cup\{u,v\}$, where $v\in B$, is a $t$-cover of $\mf$, and $\bigcap_{F\in\mb}F=T\not\subseteq G_1$. Then $t+1\le\tf\le t+2$. Suppose $S\in{[n]\c t+1}\cup{[n]\c t+2}$ is a $t$-cover of $\mf$. 
	
	Assume  $T\not\subseteq S$. 
	There exist $F_6,F_7\in\mb$ such that $F_6\cap F_7=T$ and $F_6,F_7\subseteq T\cup B\cup C$. 
	By $|S\cap F_6|\ge t$ and $|S\cap F_7|\ge t$, we have $|S|-|S\cap T|\ge2(t-|S\cap T|)$, implying that $|S\cap T|\ge2t-|S|\ge t-2$. 
	
	If $|S\cap T|=t-2$, then by $|S|-(t-2)\ge4$ and $|S|\le t+2$, we get $|S|=t+2$ and $|S\cap(B\cup C)|=|S\cap((F_6\cup F_7)\bs T)|=4$. 
	We may assume $|S\cap B|\le|S\cap C|$, which implies $B\bs S\neq\emptyset$. Pick $w\in S\cap C$. We have $|((B\cup C)\bs S)\cup\{w\}|=2(k-t)-3\ge k-t$. There exists $E\in{B\cup C\c k-t}$ such that $E\cap S=\{w\}$ and $E\cap B\neq\emptyset$. 
	Then $F_8=T\cup E\in\mb$ and $|S\cap F_8|=|S\cap T|+1<t$, a contradiction. We further conclude $|S\cap T|=t-1$ and $|S\cap(B\cup C)|\le3$. 
	
	If $S\cap B=\emptyset$, then $A\subseteq S$ follows from $|S\cap G_1|\ge t$. We also know  $|S\cap C|\le|S\bs A|\le2$, and $C\bs S\neq\emptyset$ from $|C|\ge3$. Let $I\in{B\cup(C\bs S)\c k-t}$ with $|I\cap C|=1$, and $F_9=T\cup I$. 
	Then $F_9\in\mb$ and $|F_9\cap S|=|T\cap S|<t$, a contradiction. Hence $S\cap B\neq\emptyset$. Similarly, we obtain $S\cap C\neq\emptyset$.
	
	Recall that $|S\cap(B\cup C)|\le3$. This together with $|B|=|C|\ge3$, $S\cap B\neq\emptyset$ and $S\cap C\neq\emptyset$ yields $B\bs S\neq\emptyset$ and $C\bs S\neq\emptyset$.  Then there exists $F_{10}\in\mb$ with $F_{10}\cap(S\cap(B\cup C))=\emptyset$ and $F_{10}\subseteq T\cup B\cup C$, implying that $|S\cap F_{10}|=|S\cap T|<t$, a contradiction. Hence $T\subseteq S$ and $\tttf=t$.
	
	Suppose $|S|=t+1$. Then $S\neq A\cup T$ from $|S\cap G_3|\ge t$ and $|A\cap T|=t-1$. By $T\subseteq S$, we further get $|S\cap G_i|=t-1$ for some $i\in\{1,2\}$, a contradiction.  Thus $|S|=t+2$ and $\tf=t+2$. 
	
	For $R\in{[n]\c t+2}$ with $T\subseteq R$, it is routine to check that, if $A\not\subseteq R$, then $R\in\ttf$ if and only if neither $R\cap B$ nor $R\cap C$ is empty; if $A\subseteq R$, then $R\in\ttf$ if and only if $R\cap(B\cup\{u\})\neq\emptyset$. We further conclude that $|\ttf|=(k-t)(k-t+1)+1$, as desired.

	(\ro3)	Since $|A\cup T|=t+1$, we have
	$$|\ma|={n-t-1\choose k-t-1}-{n-k-2\choose k-t-1}={n-t-1\c k-t-1}-{n-k-1\c k-t}+{n-k-2\c k-t}.$$
	By using inclusion-exclusion, we get
	$$|\mb|={n-t-1\c k-t}-2{n-k-1\c k-t}+{n-2k+t-1\c k-t}.$$
	The families $\ma$, $\mb$ and $\{G_1,G_2,G_3\}$ are pairwise disjoint. We further obtain $|\mf|=f_1(n,k,t)$.
	
	Observe that 
	\begin{equation*}
		\begin{aligned}
			&\left\{F\in{[n]\c k}: A\cup T\subseteq F,\ |F\cap(B\cup\{u\})|=1\right\}\subseteq\ma,\\
			&\left\{F\in{[n]\c k}: F\cap(A\cup T)=T,\ |F\cap B|=|F\cap C|=1\right\}\subseteq\mb.
		\end{aligned}
	\end{equation*}
	We further obtain
	\begin{align*}
		|\ma|&\ge(k-t+1)\left({n-t-2\c k-t-2}-(k-t){n-t-3\c k-t-3}\right),\\
		|\mb|&\ge(k-t)^2\left({n-t-3\c k-t-2}-2(k-t-1){n-t-4\c k-t-3}\right)\\
		&=(k-t)^2\left({n-t-2\c k-t-2}-{n-t-3\c k-t-3}-2(k-t-1){n-t-4\c k-t-3}\right)\\
		&\ge(k-t)^2\left({n-t-2\c k-t-2}-(2k-2t-1){n-t-3\c k-t-3}\right).
	\end{align*}
	These together with $|\mf|>|\mf_T|=|\ma|+|\mb|$ finish the proof of (\ro3).
\qed
\subsection{Proof of Lemma \ref{con2-maximal}}\label{verify-2}

	For convenience, set $\ms_i=\{F\in\mf:|F\cap W|=t+i\}$, where $i\in\{0,1,2\}$.

	(\ro1) By definition, each member of $\ms_2$ contains $W$, implying that $\ms_2$ is $t$-intersecting. We also know $\ms_2$ and $\ms_0\cup\ms_1$ are cross $t$-intersecting. 
	Pick $F_0,F_1\in\ms_0\cup\ms_1$. If $F_0,F_1\in\ms_0$, then $|F_0\cap F_1|=|F_0\cap F_1\cap W|+|M\bs W|\ge (t-2)+(k-t)>t$; 
	if $F_0,F_1\in\ms_1$, then $|F_0\cap F_1|\ge|F_0\cap F_1\cap W|\ge t$; if $F_0\in\ms_0$ and $F_1\in\ms_1$, then $|F_0\cap F_1|\ge|F_0\cap F_1\cap W|+1\ge t$. Now we conclude that $\mf$ is $t$-intersecting.
	
	Pick $H\in{[n]\c k}\bs\mf$. To show  $\mf$ is maximal, it is sufficient to prove that $\mf\cup\{H\}$ is not $t$-intersecting.
	
	It follows from  $H\not\in\ms_2$ that $|H\cap W|\le t+1$. 
	If $|H\cap W|<t$, then by $n>2k$ there exists $F_2\in\ms_2$ such that $|H\cap F_2|=|H\cap W|<t$. 
	If $|H\cap W|=t$, then by $H\not\in\ms_0$ and $|[n]\bs W|\ge2k-t-1>2k-2t-1$, we know $M\bs W\not\subseteq H$ and there exists $F_3\in\ms_1$ such that $|H\cap F_3|=|H\cap F_3\cap W|<t$. 
	If $|H\cap W|=t+1$, then by $H\not\in\ms_1$, we have $H\cap(M\bs W)=\emptyset$ and there exists $F_4\in\ms_0$ such that $|H\cap F_4|=|H\cap F_4\cap W|<t$. 
	We further conclude that $\mf\cup\{H\}$ is not $t$-intersecting, as desired.

	(\ro2) Since $|[n]\bs W|>2(k-t-1)$ and $|M\bs W|\ge3$,  there exist $F_5,F_6\in\ms_1$ such that $|F_5\cap F_6|=|F_5\cap F_6\cap W|=t$. Then $F_5\cap F_6\not\subseteq F_7$ for some $F_7\in\ms_0$. We have $\tf\ge t+1$. Note that $\tf\le t+2$ since $W$ is a $t$-cover of $\mf$.

	Let $T\in\ttf$. Since $n>2k$ and $|M\bs W|\ge3$, 
	if $|T\cap W|<t$, then $|T\cap F_8|=|T\cap W|<t$ for some $F_8\in\ms_2$; if $|T\cap W|=t$, then there exists $F_9\in\ms_1$ such that $|T\cap F_9|=|T\cap F_9\cap W|<t$. These contradictions yield $|T\cap W|\ge t+1$.
	
	Now $|T\cap M|\ge |T\cap W|\ge t+1$. If $|T\cap M|=t+1$, then $T\cap M=T\cap W$, and   $|T\cap F_{10}|=|T\cap F_{10}\cap W|<t$ for some $F_{10}\in\ms_0$, a contradiction. Hence $|T\cap M|\ge t+2$, which implies $|T|\ge t+2$. This together with $T\in\ttf$ and $\tf\le t+2$ yields $\tf=t+2$.
	
	By the argument above, we know $T\in{M\choose t+2}$ and $|T\cap W|\ge t+1$. It is routine to check that $\{T\in{M\c t+2}: |T\cap W|\ge t+1\}\subseteq\ttf$. Then $\ttf=\{T\in{M\c t+2}: |T\cap W|\ge t+1\}$. We further obtain $|\bigcap_{S\in\ttf}S|=|\bigcap_{S\in{W\c t+1}}S|<t$. This together with the fact that each $(t+1)$-subset of $W$ is a $t$-cover of $\ttf$ yields $\tttf=t+1$, as desired.
	
	(\ro3) A direct calculation shows that $|\ms_0|={t+2\c2}$ and $|\ms_2|={n-t-2\c k-t-2}$. We also have
	$$\ms_1=\left\{F\in{[n]\c k}: |F\cap W|=t+1\right\}\bs\left\{F\in{[n]\c k}: |F\cap W|=t+1, F\cap(M\bs W)=\emptyset\right\}$$
	and $|\ms_1|={t+2\c1}({n-t-2\c k-t-1}-{n-k-2\c k-t-1})$. Then $|\mf|=f_2(n,k,t)$ by $|\mf|=|\ms_0|+|\ms_1|+|\ms_2|$. From 
	\begin{align*}
		|\ms_1|&\ge\left|\left\{F\in{[n]\c k}: |F\cap W|=t+1,\ |F\cap(M\bs W)|=1\right\}\right|\\
		&\ge(t+2)(k-t)\left({n-t-3\c k-t-2}-(k-t-1){n-t-4\c k-t-3}\right)\\
		&\ge(t+2)(k-t)\left({n-t-2\c k-t-2}-(k-t){n-t-3\c k-t-3}\right),
	\end{align*}
	we  obtain
	$$|\mf|>|\ms_1|+|\ms_2|\ge((t+2)(k-t)+1){n-t-2\c k-t-2}-(t+2)(k-t)^2{n-t-3\c k-t-3},$$
	as desired.
\qed
\subsection{Proof of Lemma \ref{con3-maximal}}\label{verify-3}

Write $\ma_i=\{F\in{[n]\c k}:|F\cap Z|=t+i\}$ for $i=2,3,4$.

	(\ro1) Two subsets of $Z$ with sizes at least $t+2$ have intersection of size at least $t$. Then for any $F_1,F_2\in\mf$, we have $|F_1\cap F_2|\ge|(F_1\cap Z)\cap(F_2\cap Z)|\ge t$. Thus $\mf$ is $t$-intersecting.
	
	Pick $H\in{[n]\c k}\bs\mf$. We have $|H\cap Z|\le t+1$. 
	If $|H\cap Z|\le t$, then $|H\cap A_1|<t$ for some $A_1\in{Z\c t+3}$. 
	Since $|[n]\bs(Z\cup H)|\ge(2k+1)-(k+t+4)=k-t-3$, we have $F_3=A_1\cup B_1\in\ma_3$  and $|H\cap F_3|=|H\cap A_1|<t$ for some $B_1\in{[n]\bs(Z\cup H)\c k-t-3}$. 
	Suppose $|H\cap Z|=t+1$. Then $|H\cap A_2|<t$ for some $A_2\in{Z\c t+2}$. Since $|[n]\bs(Z\cup H)|\ge(2k+1)-(k+3)>k-t-2$, we have $F_4=A_2\cup B_2\in\ma_2$ and $|H\cap F_4|=|H\cap A_2|<t$ for some $B_2\in{[n]\bs(Z\cup H)\c k-t-2}$. 
	Therefore $\mf\cup\{H\}$ is not $t$-intersecting, and $\mf$ is maximal.

	(\ro2) We have $|\bigcap_{F\in\mf}F|<t$, and each $(t+2)$-subset of $Z$ is a $t$-cover of $\mf$. 
	Then $t+1\le\tf\le t+2$. 
	Suppose $S\in{[n]\c t+1}\cup{[n]\c t+2}$ with $|S\cap Z|\le t+1$. There exists $A_3\in{Z\c t+2}$ with $|S\cap A_3|<t$. 
	Since $|[n]\bs(S\cup Z)|\ge(2k+1)-(k-1)-(t+4)=k-t-2$. 
	We have $F_5=A_3\cup B_3\in\ma_2$  and $|S\cap F_5|=|S\cap A_3|<t$ for some $B_3\in{[n]\bs(S\cup Z)\c k-t-2}$.
	Therefore $\tf=t+2$ and $\ttf\subseteq{Z\c t+2}$. 
	We also have ${Z\c t+2}\subseteq\ttf$. Then $\ttf={Z\c t+2}$, and $\tttf=t+2$ follows.

	(\ro3) We obtain
	$$|\mf|=\sum_{i=2}^4|\ma_i|={t+4\c2}{n-t-4\c k-t-2}+(t+4){n-t-4\c k-t-3}+{n-t-4\c k-t-4}=f_3(n,k,t).$$
	We also have
	$$|\mf|>|\ma_2|\ge{t+4\c2}\left({n-t-2\c k-t-2}-2{n-t-3\c k-t-3}\right),$$
	as desired.
\qed

\medskip
\noindent{\bf Conflict of interest.}	
We have no known financial and personal relationships with other people or organizations that could potentially 
influence the work in this paper.

\medskip
\noindent{\bf Acknowledgment.}	
The authors would like to thank the reviewers for their careful reading of the paper and their helpful comments which led to a great improvement of the presentation of the paper. T. Yao is supported by Natural Science Foundation of Henan (262300422621). K. Wang is supported by the National Natural Science Foundation of China (12131011, 12571347) and Beijing Natural Science Foundation (1252010, 1262010).

 \medskip
 \noindent{\bf Data availability.}	
 No data was used for the research described in this paper.


\begin{thebibliography}{99}
	\bibitem{AK} R. Ahlswede and L.H. Khachatrian, The complete nontrivial-intersection theorem for systems of finite sets, J. Combin. Theory Ser. A 76 (1996) 121--138.
	\bibitem{AK2} R. Ahlswede and L.H. Khachatrian, The complete intersection theorem for systems of finite sets, European J. Combin. 18 (1997) 125--136.
	\bibitem{T1} A. Bickle, Intersecting families of $3$-sets, Australas. J. Combin. 93 (2025) 216--223.
	\bibitem{SHMC} M. Cao, M. Lu, B. Lv and K. Wang, Nearly extremal non-trivial cross $t$-intersecting
	families and $r$-wise $t$-intersecting families, European J. Combin. 120 (2024) 103958.
	\bibitem{EKR} P. Erd\H{o}s, C. Ko and R. Rado, Intersection theorems for systems of finite sets, Quart. J. Math. Oxford Ser. (2) 12 (1961) 313--320.
	\bibitem{EL1975} P. Erd\H{o}s and L. Lov\'{a}sz, Problems and results on $3$-chromatic hypergraphs and some related questions, in:  Infinite and Finite Sets (Keszthely, 1973), Vol. \uppercase\expandafter{\romannumeral2}, Colloq. Math. Soc. J\'{a}nos Bolyai, Vol. 10, North-Holland, Amsterdam, 1975, pp. 609--627.
	\bibitem{Fn} P. Frankl, The Erd\H{o}s-Ko-Rado theorem is true for $n=ckt$, in: Combinatorics (Keszthely, 1976), Vol. \uppercase\expandafter{\romannumeral1},  Colloq. Math. Soc. J\'{a}nos Bolyai, Vol. 18, North-Holland, Amsterdam, 1978, pp. 365--375.
	

	
	\bibitem{SHMT} P. Frankl, On intersecting families of finite sets, J. Combin. Theory Ser. A 24 (1978) 146--161.
	\bibitem{Frankl1980} P. Frankl, On intersecting families of finite sets, Bull. Aust. Math. Soc. 21 (1980) 363--372.
	\bibitem{Frankl2026} P. Frankl, Critically intersecting hypergraphs, European J. Combin. 132 (2026) 104286.
	\bibitem{Frankl2023} P. Frankl and A. Kupavskii, Uniform intersecting families with large covering number, European J. Combin. 113 (2023) 103747.
	\bibitem{Frankl1995} P. Frankl, K. Ota and N. Tokushige, Uniform intersecting families with covering number four, J. Combin. Theory Ser. A 71 (1995) 127--145.
	\bibitem{Frankl1996} P. Frankl, K. Ota and N. Tokushige, Covers in uniform intersecting families and a counterexample to a conjecture of Lov\'{a}sz, J. Combin. Theory Ser. A 74 (1996) 33--42.
	
	
	\bibitem{Frankl2025}  P. Frankl and J. Wang, Intersecting families with covering number three, J. Combin. Theory Ser. B 171 (2025) 96--139.
	
	\bibitem{Frudi1988} Z. F\"{u}redi, Matchings and covers in hypergraphs, Graphs Combin. 4 (1988) 115--206.
	
	
	\bibitem{SHM} A.J.W. Hilton and E.C. Milner, Some intersection theorems for systems of finite sets, Quart. J. Math. Oxford Ser. (2)  18 (1967) 369--384.
\bibitem{Kupaviskii2026} 	A. Kupavskii, Intersecting families with covering number $3$, J. Combin. Theory Ser. B 177 (2026) 216--233.
	\bibitem{L1975} L. Lov\'{a}sz, On minimax theorems of combinatorics, Mat. Lapok 26 (1975) 209--264.
	\bibitem{Moura1999} L. Moura, Maximal $s$-wise $t$-intersecting families of sets: kernels, generating sets, and enumeration, J. Combin. Theory Ser. A 87 (1999) 52--73.
		\bibitem{Polcyn2017} J. Polcyn and A. Ruci\'{n}ski, A hierarchy of maximal intersecting triple systems, Opuscula Math. 37 (2017) 597--608.
	
	
	\bibitem{2601191} C. Shi, P. Frankl and J. Qian, On non-empty cross-intersecting families, Combinatorica 42 (2022) 1513--1525.
	

	\bibitem{Wn}R.M. Wilson, The exact bound in the Erd\H{o}s-Ko-Rado theorem, Combinatorica 4 (1984) 247--257.
\end{thebibliography}
\end{document}